\documentclass[12pt,twoside,reqno]{amsart}
\usepackage[colorlinks=true,citecolor=blue]{hyperref}
\usepackage{mathptmx, amsmath, amssymb, amsfonts, amsthm, mathptmx, enumerate, color,mathrsfs}
\usepackage{xcolor}
\usepackage{bm}

\usepackage{graphicx}

\usepackage{multirow}
\usepackage{epstopdf}
\usepackage{multicol}
\usepackage{algorithm}
\usepackage{algorithmic}
\usepackage{epstopdf}
\usepackage{float}

\newtheorem{theorem}{Theorem}[section]
\newtheorem{lemma}[theorem]{Lemma}
\newtheorem{proposition}[theorem]{Proposition}
\newtheorem{corollary}[theorem]{Corollary}

\newtheorem{property}[theorem]{Property}

\theoremstyle{definition}
\newtheorem{definition}[theorem]{Definition}

\newtheorem{example}[theorem]{Example}

\newtheorem{remark}[theorem]{Remark}
\numberwithin{equation}{section}

\DeclareMathSymbol{\perp}{\mathrel}{symbols}{"3F}
\DeclareMathSymbol{\bot}{\mathord}{symbols}{"3F}

\def\R{\mathbb R}

\everymath{\displaystyle}

\usepackage{hyperref}

\DeclareMathOperator{\argmax}{argmax}

\newcommand{\e}{\mathbf{e}}
\newcommand{\A}{\mathbf{A}}
\newcommand{\B}{\mathbf{B}}
\newcommand{\I}{\mathbf{I}}
\newcommand{\x}{\mathbf{x}}

\newcommand{\rowp}[2]{\ensuremath{{#1}_{#2}^+}}
\newcommand{\rowm}[2]{\ensuremath{#1_#2^-}}
\newcommand{\rowmm}[3]{\ensuremath{#1_{#2#3}^-}}
\newcommand{\rowpp}[3]{\ensuremath{#1_{#2#3}^+}}

\usepackage{tikz,pgfplots,subcaption}
\pgfplotsset{compat=1.18}
\usetikzlibrary{intersections,calc}
\usepgfplotslibrary{fillbetween}
\usepackage{caption}

\begin{document}
\setcounter{page}{1}

\vspace*{1.8cm}
\title[Localization of complementarity eigenvalues]
{Localization of complementarity eigenvalues}
\author[A. Sasaki, S. Demassey, V. Sessa]{Antonio Sasaki$^{1}$, Sophie Demassey$^{1}$, Valentina Sessa$^{1}$}
\maketitle
\vspace*{-0.6cm}

\begin{center}
{\footnotesize

$^1$Centre de Math\'ematiques Appliqu\'ees (CMA), \'Ecole nationale sup\'erieure des Mines de Paris (Mines Paris), Universit\'e Paris Sciences et Lettres (PSL), Sophia Antipolis, France\\
\vspace{0.2cm}

Dedicated to Professor Alfredo N. Iusem on the occasion of his 75th birthday.

}\end{center}

\vskip 4mm {\footnotesize \noindent {\bf Abstract.}
Let $\mathbf{A},\mathbf{B}$ be symmetric $n\times n$ real matrices with $\mathbf{B}$ positive definite and strictly diagonally dominant. We derive two localization sets for the complementarity eigenvalues of $(\A,\B)$, the tightest one assuming additionally that $\mathbf{A}$ is copositive. This extends He-Liu-Shen sets to the case where $\B$ is not the identity. Moreover, we compare the computable bounds obtained from these new sets with the extreme classical generalized eigenvalues.

 \noindent {\bf Keywords.} Complementarity problems; Gershgorin-type localization sets. 

 \noindent {\bf 2020 Mathematics Subject Classification.}
90C33, 15A18.}

\renewcommand{\thefootnote}{}
\footnotetext{
E-mail address: \{antonio.sasaki,  sophie.demassey, valentina.sessa\}@minesparis.psl.eu
}

\section{Introduction}
Let $\A,\B$  be symmetric $n\times n$ real matrices and $\B$ be positive definite. The symmetric \emph{eigenvalue complementarity problem} (EiCP) is to find $\x\in\R^n \setminus \{\mathbf{0}\}$ and $\lambda\in\R$ satisfying:
% \begin{equation}\label{eq:EiCP}
% \e ^\top \x =1,\quad \x\ge \mathbf{0},\quad
% (\A-\lambda \B)\x\ge \mathbf{0},\quad
% \x^\top(\A-\lambda \B)\x=0,
% \end{equation}
\begin{subequations}\label{eq:EiCP}
\begin{gather}
(\A-\lambda \B)\x\ge \mathbf{0} \label{eq:EiCP_a}\\
\x\ge \mathbf{0} \label{eq:EiCP_b}\\
\x^\top(\A-\lambda \B)\x=0 \label{eq:EiCP_c}\\
\e ^\top \x =1 \label{eq:EiCP_d}, 
\end{gather}
\end{subequations}
% \sd{Do  we really need to introduce notation $\ip{\cdot}{\cdot}$ ? just write $\x^\top (\A-\lambda \B)\x$ just like $e^\top x=1$ ? }
where $\e := (1,\dots,1)^\top \in \R^n$ is the all-ones vector and the inequalities are meant to be component-wise. For a pair $(\x,\lambda)$ satisfying \eqref{eq:EiCP}, $\lambda \in \R$ is a \emph{complementarity eigenvalue} and $\x \in \R^n$ is a \emph{complementarity eigenvector}. The orthogonality condition~\eqref{eq:EiCP_c} together with the nonnegativity requirements~\eqref{eq:EiCP_a} and~\eqref{eq:EiCP_b} imply that $x_i = 0$ or $w_i=0$ with $\mathbf{w}= (\A-\lambda \B)\, \x$
%$(\A-\lambda \B) ^\top _i \, \x = 0$ 
for all $i = 1,...,n$.
%, where $\mathbf{M}_i ^\top$ indicate the $i$-th row of the matrix $\mathbf{M}$.

Seeger introduced this problem in~\cite{Seeger1999} for the case when $\B$ is equal to the identity matrix. In this context, this problem is often referred to as the \emph{Pareto eigenvalue problem} because the complementarity spectrum is then attached to a single matrix $\A$ under the nonnegativity cone (also known as Pareto cone). Queiroz et al.~\cite{Queiroz2004} subsequently extended the formulation to general symmetric pairs $(\A,\B)$, motivated by the stability analysis of mechanical systems with frictional contacts.
The practical importance of EiCP is further underscored by applications in copositive matrix analysis \cite{HiriartUrruty2010}, asymptotic studies of Fucik curves \cite{Holubova2015}, graph theory \cite{Fernandes2017, Seeger2018}, and linear dynamical systems governed by complementarity conditions \cite{Seeger1999}.

Theoretical and numerical properties of the symmetric EiCP have been investigated in recent years, as well as different generalizations~\cite{bras2015,BIJu,FNZh,iusem2016,IJSS2%,See2
}.
%many interesting theoretical and numerical results have been presented. 
%\sd{While the number of standard eigenvalues of $n\times n$ matrices is lower than or equal to $n$, it is shown in~\cite{SeVi} that the number of complementary eigenvalues
%grows exponentially with $n$.}
%The number of solutions to an EiCP was studied in~\cite{SeVi}, where the authors showed that this number grows exponentially with the dimension $n$, of EiCP, 
%which is much higher than the traditional \emph{eigenvalue problem}, where this number is lower than or equal to $n$. 
%\sd{Precisely, an EiCP has up to $2^n-1$ solutions $\lambda$, and up to $n2^{n-1}$ in the non-symmetric case~\cite{Seeger1999}.}
%has at most $n2^{n-1}$ distinct \sd{solutions $\lambda$}~\cite{Seeger1999}, and for the symmetric case, this number drops to $2^n-1$. 
%\sd{[the non-symmetric case has not been defined ?]}
%
The numerical resolution of EiCP is often based on its reformulation as a more manageable problem, for which specialized algorithms can be designed.
For instance, in~\cite{Judice2021}, the EiCP is formulated as a difference-of-convex program over the standard simplex, then solved with an alternating direction method of multipliers, which globally converges to a solution of the symmetric EiCP.
%Then, an Alternate Direction Method of Multipliers is implemented, which has been shown to be globally convergent to a solution of the symmetric EiCP. 
In~\cite{Fukushima2020}, the symmetric  EiCP is reformulated as
finding a stationary point to the minimization
%a minimization problem whose objective function consists 
of the sum of a convex function and a smooth but nonconvex function. 
%To find a stationary point of this problem, which corresponds to the solution of the EiCP, an 
An algorithm, that solves a sequence of convex quadratic problems obtained by linearizing the nonconvex term,
%in the objective function, 
is shown to converge  to a solution of the EiCP under mild assumptions.
% A proof of global convergence to a solution of the EiCP is given under mild assumptions.
%
A different reformulation, as a fractional program, is to minimize the ratio of two quadratic functions, namely the generalized Rayleigh quotient, over the standard simplex~\cite{JRRS,Queiroz2004}.
%A different reformulation can be derived considering that computing a solution of a symmetric EiCP is also equivalent to finding a stationary point of the minimization problem of the ratio of two quadratic functions (i.e., the generalized Rayleigh function) over the standard simplex~\cite{JRRS,Queiroz2004}. 
In~\cite{judice2022}, a stationary point %of a fractional quadratic program over the standard simplex 
is computed by using an efficient implementation of the 
% well-known 
Dinkelbach’s method \cite{Schaible1976} 
%applied to a special fractional linear quadratic program obtained by 
after linearizing the numerator of the objective function around a current iterate.
In~\cite{AdSe}, the EiCP is formulated as an equivalent nonsmooth system of equations based on complementarity functions~\cite{FaPa} and solved using a semi-smooth Newton method.
%Then,  a Semi-smooth Newton (SN) method is used to find a zero for this system. This approach has garnered significant interest in solving both nonsymmetric and symmetric EiCPs due to its fast convergence properties. However, global convergence is not guaranteed, and the algorithm may fail in many instances. A different reformulation of the EiCP based on finding the eigenvalue of a nonlinear problem was proposed in~\cite{adly2013}, without requiring complementarity functions.
In~\cite{adly2023}, the EiCP is formulated as a nonlinear system of equations,  solved with  two versions of the interior point method, which appear to be numerically more efficient than the smoothing method on medium-sized instances ($n\approx 100$), but not on smaller instances.
%are implemented. 
%The numerical results show that, for small EiCPs, the smoothing method outperforms the %proposed 
%interior-point methods. However, for medium-sized instances with an average size of 131, the interior-point methods are more effective.
% \sd{[Could you say more about how all these method compare ?]}
% \vs{VS: Unfortunately, do to the lack of a "true" benchmark on EiCP, there is not a paper that compares all of them. This is something that I had in mind to do after my HDR.  }

%Some generalizations of EiCP have been introduced in recent years; the interested reader can refer to the following works~\cite{AdRa,bras2015,BIJu,FNZh,iusem2016,IJSS2,See2}.

EiCP~\eqref{eq:EiCP} extends the generalized eigenvalue problem~\cite{HornJohnson2013}, defined as $\A\x =\lambda\B\x$, that is \eqref{eq:EiCP_a} satisfied at equality and dropping \eqref{eq:EiCP_b} and  \eqref{eq:EiCP_c}. 
% A fruitful way to view EiCP is through generalized eigenpairs \sd{[not defined]} of principal submatrices \sd{[not defined]}. Any generalized eigenvalue of $(\A,\B)$ that admits a nonnegative eigenvector is a complementarity eigenvalue \sd{[is it a definition or a result ? ref anyway]}. Conversely, 
Note that if $(\x,\lambda)$ solves \eqref{eq:EiCP} with $x_i>0$, for all $i=1,...n$, the orthogonality constraint~\eqref{eq:EiCP_c} forces $\A\x-\lambda\B\x=\mathbf{0}$. More generally, if a subvector  $\x_S>\mathbf{0}$ on an index set $S \subset \{1,\dots,n\}$, then $(\x_S,\lambda)$ is a generalized eigenpair of the submatrices $(\A_{SS},\B_{SS})$ \cite{Fernandes2014}. This submatrix perspective suggests a constructive path: by scanning principal submatrices and computing classical generalized eigenpairs, one can enumerate complementarity eigenvalues~\cite{Seeger1999}.
This procedure is computationally viable only for very small instances ($n \leq 15$). Unlike the generalized eigenvalue problem, which has at most $n$ solutions, the system \eqref{eq:EiCP} may exhibit exponentially many complementarity eigenpairs as $n$ grows, at most $2^n-1$, then exhaustive enumeration becomes prohibitive \cite{Fernandes2014}. 

% This \vs{VS: This what? I would add something about numerical algorithms needing bounds} motivates inexpensive spectrum localization sets \sd{[not defined]} for EiCP, namely upper and lower bounds for $\lambda$ that are cheap to compute and informative enough to prune candidates, certify outputs, or guide solvers. \sd{[to solve what ?]} 
%
% I will take care of this part
% Modern state-of-art algorithms for symmetric EiCPs (ADMM \cite{Judice2021}, SPL \cite{Fukushima2020}, DC programming \cite{Niu2019}, and spectral BAS \cite{Bras2017}) can compute complementarity eigenpairs directly from \eqref{eq:EiCP}, \sd{[you just said that the number of solutions is exponential ?]} yet their effectiveness benefits from such bounds \sd{[how ?]}, especially when the spectrum is large. 

Recently, He et al. \cite{He2023} resolved an open question of Seeger \cite{Seeger1999}, regarding the location of the complementarity eigenvalues when $\B$ is the identity matrix. The authors derived computable localization sets, as intervals in $\mathbb{R}$, %for complementarity eigenvalues 
that depend only on row sums of $\A$, in a way related to the Gershgorin circle theorem \cite{HornJohnson2013} for ordinary eigenvalues. These results provide explicit algebraic bounds for $\lambda$.

In this paper, we address the same question for the symmetric EiCP \eqref{eq:EiCP} with $\B$ strictly diagonally dominant. It is worth mentioning that generalizing $\B$ from the identity matrix to any $\B$  positive definite and strictly diagonally dominant represents a contribution in this context. Although it may appear restrictive, the positive definiteness of $\B$ is commonly assumed in EiCPs to ensure the existence of a solution~\cite{Fernandes2014, Fukushima2020, Judice2021, judice2022, Queiroz2004}, and the diagonal dominance assumption is also standard and frequently employed in the localization theory of various eigenvalue classes via Gershgorin-type sets~\cite{kostic2015, nakatsukasa2011, singh2023}.  
% \cite{kostic2015, kostic2009, kostic2018, kostic2012, nakatsukasa2011, singh2023}
Extending the approach in \cite{He2023}, we present localization sets, based on either single rows or pairs of rows, under the extra assumption that $\A$ is copositive in the latter case. When $\B$ is the identity, our formulas reduce to those in \cite{He2023}, thereby recovering the Pareto setup. We also show that the two-row refinement provides tighter sets than the one-row analysis and that there is no dominance between the bounds provided by these localization sets and the extreme (largest and smallest) generalized eigenvalues. 

This paper is organized as follows: first we introduce the necessary notation and preliminary concepts. In Section~\ref{sec2}, we derive the one-row localization set $K_1$ and the two-row localization set $K_2$ under the additional hypothesis that $\mathbf{A}$ is copositive. We then prove in Section~\ref{sec3} that $K_2$ is contained in $K_1$. Section~\ref{sec:lbub} extracts computable lower and upper bounds from these sets. Finally, Section~\ref{section5} compares our localization sets with the standard generalized eigenvalue spectrum, and Sections~\ref{secd} and~\ref{secc} provide concluding remarks and discuss potential extensions.

% We extend the approac of \cite{He2023} to the general EiCP \eqref{eq:EiCP}. Under the standing assumption that $\B$ is strictly diagonally dominant, and with $\A$ copositive in the two-rows analysis, we obtain one-row bounds and a two-rows refinement that localize all complementarity eigenvalues. When $\B=\I$, our formulas reduce to those of He-Liu-Shen, thereby recovering the Pareto scheme. 
% The resulting enclosures appear in Theorem~\ref{thm:one-row-theorem} and Theorem~\ref{thm:two-row-theorem}. 

\section*{Notation and Preliminaries}
% Before proceeding, we fix the notation used to ensure the manuscript is 
% self-contained. \vs{Papers are almost never self-contained...}

We indicate with $[n]$ the set of natural numbers between $1$ and $n$, i.e., $[n]:= \{1,...,n\}$.  Let $\R^n$, $\R^n_+$ and $\mathbb{S}^n$ denote the $n$-dimensional Euclidean space, the nonnegative orthant, and the space of $n \times n$ real symmetric matrices, respectively. 

For a vector $\x \in \R^n$, we denote its $j$-th component by $x_j,\: j \in [n]$. The element $(i,j)$ of a matrix $\mathbf{M} \in \R^{n\times n}$ is denoted as $m_{ij}$, and $\mathbf{M} ^\top _i$ represents the $i$-th row of $\mathbf{M}$.  

The $j$-th standard basis vector in $\R^n$ is denoted by $\e_j$. We denote by $\e \in \R^n$ the $n$-dimensional vector of ones and by $\mathbf{E} = \e \e ^\top \in \mathbb{S}^n$ the matrix with all elements being equal to one. 

\begin{definition}[Diagonal dominance]\label{def: diagonal dominance}
Let $\mathbf{M}\in\mathbb{R}^{n\times n}$. 
We say that $\mathbf{M}$ is \emph{diagonally dominant} (DD) if $|m_{ii}| \,\ge\, \sum_{j\neq i} |m_{ij}| \ \text{for every } i\in[n]$. It is \emph{strictly diagonally dominant} ($\operatorname{SDD}$) if $|m_{ii}|~\,>\,~\sum_{j\neq i} |m_{ij}| \ \text{for every } i\in[n]$.
\end{definition}

\begin{definition}[Positivity]\label{def:positivity}
Let $\mathbf{A}\in\mathbb{S}^n$. We say that $\mathbf{A}$ is \emph{positive semidefinite} (PSD) if $\mathbf{x}^\top \mathbf{A}\,\mathbf{x}~\ge~0$ for all $\mathbf{x}\in\mathbb{R}^n$. It is \emph{positive definite} ($\operatorname{PD}$) if $\mathbf{x}^\top \mathbf{A}\,\mathbf{x} \,>\, 0 \ \text{for all } \mathbf{x}\in\mathbb{R}^n \text{ with } \mathbf{x}\neq \mathbf{0}$.
\end{definition}

\begin{definition}[Copositivity]\label{def: copositivity}
Let $\mathbf{A}\in\mathbb{S}^n$. We say that $\mathbf{A}$ is \emph{copositive} if $\mathbf{x}^\top \mathbf{A}\,\mathbf{x}\;\ge\; 0$ for all  $\mathbf{x}\in\mathbb{R}^n$ with $\x \geq \mathbf{0}$.
It is \emph{strictly copositive} if $\mathbf{x}^\top \mathbf{A}\,\mathbf{x}\;>\; 0 \ \text{for all } \mathbf{x}\in\mathbb{R}^n$ with $\x>\mathbf{0}$.
\end{definition}

Clearly, any nonnegative or positive definite matrix is copositive. 
It is also known that symmetric SDD matrices with positive diagonal elements are PD (see Thm. 6.1.10 in~\cite{HornJohnson2013}).
%\vs{yes}
%It is also known that symmetric SDD matrices are PD~\cite{??} [è vero ? then should we remove the ref to B PD when we assume B SDD ?]}
%\vs{you still need that the diagonal elements are positive.}

\begin{property} \label{pro:1}
$\operatorname{EiCP}(\A, \B)$ has a solution $(\x, \lambda)$ if and only if $\operatorname{EiCP}(\A + \mu \B, \B)$ has a solution $(\x,\lambda +\mu)$ with $\mu \geq 0$. 
\end{property}
Note that since $\B$ is $\operatorname{PD}$, then for some $\mu>0$ the matrix $\A+\mu \B$ is also $\operatorname{PD}$~\cite{HornJohnson2013}. So, we can assume without loss of generality that $\A$ is a $\operatorname{PD}$ (or at least copositive) matrix in EiCP \eqref{eq:EiCP}. %Note that by \eqref{eq:rayquo} $\lambda > 0$ in any solution of EiCP.\\

\begin{property}\label{prop:nonnegla}
Let $\mathbf{A}, \B \in\mathbb{S}^n$ and $\B$ be positive definite. If $\A$ is a copositive matrix, then in any solution $(\x,\lambda)$ of the $\operatorname{EiCP}(\A,\B)$, the complementarity eigenvalue $\lambda$ is nonnegative.
\end{property}

A solution $(\x,\lambda)$ of the EiCP satisfies the orthogonality condition~\eqref{eq:EiCP_c}, then $\x \geq \mathbf{0}$ and 
\begin{equation}\label{eq:Rayquo}
\lambda~=~\frac{\x^\top \A \x}{\x^\top \B \x}
\end{equation}
is the 
\textit{generalized Rayleigh quotient} of $(\A,\B)$~\cite{Queiroz2004}. If $\A$ is copositive and $\B$ is PD, then by Definitions~\ref{def:positivity}
and \ref{def: copositivity}, we obtain the result in Property~\ref{prop:nonnegla}. 

\section{Localization sets for EiCP solutions}\label{sec2}

If $\A$ and $\B$ are diagonal matrices, the complementarity eigenvalues can be located easily. Indeed, they are the points ${a_{ii}}/{b_{ii}},\: i \in [n]$, in the real space~\cite{Seeger1999}.
As done for Gershgorin circle theorem~\cite{HornJohnson2013} for the classical eigenvalue problem in the complex space, we consider localizing the complementarity eigenvalues with respect to the points ${a_{ii}}/{b_{ii}}$ when the matrices $\A$ and $\B$ are not diagonal. The first localization set examines how far the matrices are to be diagonal, by regarding the off-diagonal entries row-by-row.
% For any matrix $M$, we can write $M=D+C$, where $D = diag(m_{11},...,m_{nn})$ and $C=M-D$ has zero diagonal elements. 

% \sd{[Please rephrase: the idea is just to give an intuition of the theorem] Related to the Gershgorin circle theorem, the first localization set examines how far the matrices are to be diagonal, by regarding the off-diagonal entries row by row.}

\begin{definition}\label{def:rpm}
For any matrix $\mathbf{M}\in\R^{n\times n}$, and row $i\in[n]$, the row sums of positive and negative off-diagonal entries and the associated diagonal shifts are denoted as follows:
%$r_i^\pm$ and the diagonal shifts $m_i ^\pm$ obtained by adjusting $m_{ii}$ by these sums
\begin{equation*}\label{eq:rpm}
r_i^+(\mathbf{M}):=\sum_{j\ne i}\max\{m_{ij},0\},\quad 
r_i^-(\mathbf{M}):=-\sum_{j\ne i} \min\{m_{ij},0\}.
\end{equation*}
\begin{equation*}\label{eq:rpm2}
m_i ^+ := 
\, m_{ii} \, + \, r_i^+(\mathbf{M}),
\quad  m_i ^- :=\, m_{ii} \, - \, r_i^-(\mathbf{M}).
\end{equation*}    
\end{definition}

Note that $m_i ^+ > 0$ for matrices that are positive definite. However, the sign of $m_i ^-$ is not determined a priori for these classes of matrices.

\begin{theorem}[One-row localization]\label{thm:one-row-theorem}
Let $\A,\B\in\mathbb{S}_n$, and assume that $\B$ is positive definite and strictly diagonally dominant.
%$\operatorname{PD}$ and $\operatorname{SDD}$.
Let $\lambda \in \R$ be a complementarity eigenvalue of $(\A,\B)$, then
\[
\lambda\ \in\ %\biggl \{ 
\bigcup_{i\in [n]} \, \left[
\min\left(\frac{\rowm{a}{i}}{\rowm{b}{i}},\frac{\rowm{a}{i}}{\rowp{b}{i}}\right)\!,
\, \max\left(\frac{\rowp{a}{i}}{\rowm{b}{i}},\frac{\rowp{a}{i}}{\rowp{b}{i}}\right)
\right] \, %\biggr \} 
=: \, {K}_1.
\]
%given, for any matrix $M\in \mathbb{S}_n$ and row $i\in[n]$: \begin{equation*}\label{eq:rpm}
%m_i^+= m_{ii} + \sum_{j\in[n]\setminus  i}\max\{m_{ij},0\},\ 
%m_i^- = m_{ii} + \sum_{j\in[n]\setminus i} \min\{m_{ij},0\}
%\end{equation*}
\end{theorem}

To prove this theorem, we make use of the following result.

\begin{lemma}\label{lemma:utility}
Let $\mathbf{M}\in\mathbb{R}^{n\times n}$, $\x\in\mathbb{R}^n$ with $\x\ge\mathbf{0}$ and any $p\in\arg\max_{i\in[n]} x_i$, then:
\begin{itemize}
    \item[$(i)$] $\rowm{m}{p}\,x_p \ \le\ \mathbf{M}_p^\top \x \ \le\ \rowp{m}{p}\,x_p$
    \item[$(ii)$] if $\mathbf{M}$ is strictly diagonally dominant and $m_{ii}\ge 0$, then
$m_i^+ \ge m_i^- > 0.$
\end{itemize}
\end{lemma}

\begin{proof}
Since $x_p\ge x_j\ge 0$, for $j\ne p$,
\[
-\;r_p^-(\mathbf{M})\,x_p \leq \sum_{\substack{j\ne p \\
m_{pj} < 0}} m_{pj}x_j \le\ \sum_{j\ne p} m_{pj}x_j \leq  \sum_{\substack{j\ne p \\
m_{pj} > 0}} m_{pj}x_j \le\ r_p^+(\mathbf{M})\,x_p.
\]
Adding $m_{pp}x_p$ yields $(i)$.
If $\mathbf{M}$ is $\operatorname{SDD}$ and $m_{ii}\ge 0$, then
$m_{ii}=|m_{ii}|>\sum_{j\ne i}|m_{ij}|=r_i^+(\mathbf{M})+r_i^-(\mathbf{M})\ge r_i^-(\mathbf{M})\ge 0$, and $m_i^+ \ge m_i^-=m_{ii}-r_i^-(\mathbf{M})>0$.

\end{proof}

\begin{proof}[Proof of Theorem~\ref{thm:one-row-theorem}.]
Let $(\mathbf{x},\lambda)$ be a solution of $\operatorname{EiCP}(\A,\B)$. Then $\mathbf{x}\geq \mathbf 0$, and $\e^\top \x = 1$ and there exists $p$ with $x_p=\max_{i\in[n]} x_i>0$. With $\mathbf{w}:=(\mathbf{A}-\lambda\mathbf{B})\mathbf{x}\ge0$ and $\x^\top \mathbf{w}=0$, all terms $x_i w_i$, for $i\in [n]$, are nonnegative and sum to zero. Hence $x_p w_p=0$, therefore $w_p=0$ and $\lambda = {\A^\top_p\x}/{\B^\top_p\x}$.

Lemma \ref{lemma:utility} $(i)$ provides bounds to the operands:
$\rowm{a}{p}x_p \le \mathbf{A}_p^\top \x \le \rowp{a}{p}x_p$
 and $\rowm{b}{p}x_p \le \mathbf{B}_p^\top \x \le~\rowp{b}{p}x_p$, with the bounds on the denominator being both positive since $\B$ is $\operatorname{SDD}$ and $\B$ is PD then $b_{pp}=\mathbf{e}_p^\top \B \mathbf{e}_p >0$. Thus, $\rowm{a}{p}x_p/\mathbf{B}_p^\top \x \le \lambda \le \rowp{a}{p}x_p/\mathbf{B}_p^\top \x$. 

The map $u\rightarrow v/u$ is monotonic on $\R_+$, then $\rowm{a}{p}x_p/\mathbf{B}_p^\top \x\ge\min(\rowm{a}{p}x_p/\rowm{b}{p}x_p, \rowm{a}{p}x_p/\rowp{b}{p}x_p)=\min(\rowm{a}{p}/\rowm{b}{p}, \rowm{a}{p}/\rowp{b}{p})$, and, similarly, 
$\rowp{a}{p}x_p/\mathbf{B}_p^\top \x\le\max(\rowp{a}{p}/\rowm{b}{p}, \rowp{a}{p}/\rowp{b}{p})$ as $x_p\leq 1$. 

The localization set ${K}_1$ is expressed as the union of such intervals for every possible value of $p\in[n]$, all intervals being well defined since $b_i^+\ge b_i^-> 0$ for all $i\in[n]$.

\end{proof}

%The previous theorem is similar to the Gershgorin Circle Theorem in the sense that it associates a complementarity eigenvalue with one row, both in matrices $\A$ and $\B$, corresponding to the largest row of the eigenvector.

\begin{corollary}\label{cor:onecop}
Let $\A,\B\in\mathbb{S}_n$. Assume that $\B$ is positive definite and strictly diagonally dominant, and $\A$ is copositive. Let $\lambda \in \R$ be a complementarity eigenvalue of $(\A,\B)$, then
\[
\lambda\ \in\ %\biggl \{ 
\bigcup_{i\in [n]} \, \left[
\max\left(0, \ \frac{\rowm{a}{i}}{\rowp{b}{i}}\right)\!,
\, \frac{\rowp{a}{i}}{\rowm{b}{i}}\right] \, %\biggr \} 
=: \, {K}'_1.
\]
\end{corollary}
\begin{proof}
The copositivity of the matrix $\A$ allows refining the localization set $K_1$ by considering the following properties.
\begin{enumerate}[(i)]
    \item Property~\ref{prop:nonnegla} establishes a trivial lower bound  of zero for $\lambda$. The lower bounds of each interval in $K_1$ depend on the sign of $a_i^-$. With a negative  $a_i^-$, the corresponding lower bound is also negative and thus dominated by zero. On the contrary, a tighter lower bound can be obtained when $a_i^- >0$. Since the copositivity of the matrix alone does not induce a sign on $a_i^-$, we redefine the lower bound of each interval in $K'_1$.
    \item Since $a_i^+ \geq 0$ and $b_i^+\geq b_i^- >0$ the upper bound of each interval in $K_1$ is simply $a_i^+/b_i^-$.    
\end{enumerate}
\end{proof}

% \sd{[Again rephrase to give some intuition, and insists on the new assumption A copositive] For the second localization set, we now examine how the off-diagonal entries combine for each pair of rows of the matrices.}

We now introduce an alternative localization set, which considers pairs of rows $i,j \in [n]$ of the matrices $\A$ and $\B$. As before, we compute quantities based on the off-diagonal entries. Moreover, to derive this new localization set, we make the additional assumption that $\A$ is copositive, then by Property~\ref{prop:nonnegla} one has $\lambda \geq 0$. Note that the copositivity hypothesis of $\A$ is not restrictive since it can be enforced by a shift, as shown in Property~\ref{pro:1}.

We begin with the following definitions. 

%\begin{equation*}
%s_{ij}^\pm({\A}, \B) \, := \, a_{ii}b_{jj}+a_{jj}b_{ii}+r_i^\pm(\A)r_j^\mp(\B)+r_j^\pm(\A)r_i^\mp(\B) 
%\end{equation*}
%\begin{equation*}
%s_{ij}^-({\A}, \B) \, := \, s_{ij}^+(\B, \A) \, = \, a_{ii}b_{jj}+a_{jj}b_{ii}+r_i^+(\B)r_j^-(\A)+r_j^+(\B)r_i^-(\A) 
%\end{equation*}

\begin{definition}\label{def:bij1}
For any $\mathbf{M} \in \R^{n\times n}$, and two rows $i,j\in[n]$, let us denote: 
\begin{align*}
{m}_{ij} ^+ \, &:= \, m_{ii}m_{jj}-r_i^+(\mathbf{M})r_j^+(\mathbf{M})\\
{m}_{ij} ^- \, &:= \, m_{ii}m_{jj}-r_i^-(\mathbf{M})r_j^-(\mathbf{M}).
\end{align*}  
\end{definition}

\begin{definition}\label{def:bij}
For any pair $\A , \B \in \R^{n\times n}$ and two rows $i,j\in[n]$, let us denote:
%\begin{equation}
% \begin{align*}
% {b}_{ij} ^+ \, := \, b_{ii}b_{jj}-r_i^+(\B)r_j^+(\B), \quad &s_{ij}^+({\A}, \B) \, := \, a_{ii}b_{jj}+a_{jj}b_{ii}+r_i^+(\A)r_j^-(\B)+r_j^+(\A)r_i^-(\B)   
% \end{align*}
% %\end{equation}
% \begin{align*}\
% {b}_{ij} ^- \, := \, b_{ii}b_{jj}-r_i^-(\B)r_j^-(\B), \quad &s_{ij}^-({\A}, \B) \, := \, a_{ii}b_{jj}+a_{jj}b_{ii}+r_i^-(\A)r_j^+(\B)+r_j^-(\A)r_i^+(\B)
% \end{align*}
% \begin{equation*}\label{eq:pols}
% P^{\, \operatorname{up}}_{ij}(y)
% :=\rowmm{b}{i}{j}\,y^2
% -s_{ij}^+({\A}, \B)\,y
% +\rowpp{a}{i}{j}, \quad P^{\, \operatorname{low}}_{ij}(y)
% :=\rowpp{b}{i}{j}\,y^2
% -s_{ij}^-({\A}, \B)\,y
% +\rowmm{a}{i}{j}\quad\forall y\in\mathbb{R}.
% \end{equation*}  
\begin{gather*}
s_{ij}^+({\A}, \B) \, := \, a_{ii}b_{jj}+a_{jj}b_{ii}+r_i^+(\A)r_j^-(\B)+r_j^+(\A)r_i^-(\B) \\
s_{ij}^-({\A}, \B) \, := \, a_{ii}b_{jj}+a_{jj}b_{ii}+r_i^-(\A)r_j^+(\B)+r_j^-(\A)r_i^+(\B) \\
P^{\, \operatorname{up}}_{ij}(y) \;
:=\rowmm{b}{i}{j}\,y^2
-s_{ij}^+({\A}, \B)\,y
+\rowpp{a}{i}{j}, \quad P^{\, \operatorname{low}}_{ij}(y)
:=\rowpp{b}{i}{j}\,y^2
-s_{ij}^-({\A}, \B)\,y
+\rowmm{a}{i}{j}\quad\forall y\in{\mathbb{R}}.
\end{gather*}  
\end{definition}

\begin{theorem}[Two-row localization]\label{thm:two-row-theorem}
Let $\mathbf A,\mathbf B\in\mathbb S_n$. Assume that $\B$ is positive definite and strictly diagonally dominant, and $\A$ is copositive. Let $\lambda \in \R$ be a complementarity eigenvalue of $(\A,\B)$, then
$$\ \lambda\in  
 \bigcup_{\substack{i,j\in[n]\\ i\ne j}}\,\left[\,\max\{0, \, \min\{y \in \R \mid P^{\, \operatorname{low}}_{ij}(y) =0\}\},\ \max\{y \in \R \mid P^{\, \operatorname{up}}_{ij}(y) =0\}\right] 
 %\biggr\} 
 =: {K}_2.$$
\end{theorem}

The proof of Theorem~\ref{thm:two-row-theorem} makes use of the following results.

\begin{lemma}\label{lem:two-rows-bounds}
Let $\mathbf{M} \in \mathbb{S}_n$ and $\mathbf{x}\in\mathbb{R}^n$ with $\mathbf{x}\ge \mathbf{0}$ and $\e^\top \x = 1$, and the two largest elements
\[
p\in \argmax_{i\in[n]} x_i
\: \mbox{ and }
q\in\argmax_{i\in[n]\setminus\{p\}} x_i \,
\]
%so that $x_p\ge x_q\ge x_j$ for all $j\notin\{p,q\}$. T
then the following bounds hold for $i=p$ and for $i=q$:
\begin{equation*}
%\begin{gathered}   
    -\,r_i^-(\mathbf{M})\,x_p x_q \;\le\; \sum_{j\ne i} m_{ij}\,x_i x_j \;\le\; r_i^+(\mathbf{M})\,x_p x_q, %\\   
%    -\,r_q^-(\mathbf{M})\,x_p x_q \;\le\; \sum_{j\ne q} m_{qj}\,x_q x_j \;\le\; r_q^+(\mathbf{M})\,x_p x_q.
%   \end{gathered}
\end{equation*}
\end{lemma}

\begin{proof}
Because $\mathbf{x}\ge \mathbf 0$ and $x_p\ge x_q\ge x_j$ for $j\notin\{p,q\}$, we have for the $p$-row
\[
\sum_{j\ne p} m_{pj}\,x_p x_j
=
\sum_{\substack{j\ne p\\ m_{pj}> 0}} m_{pj}\,x_p x_j
+\sum_{\substack{j\ne p\\ m_{pj}< 0}} m_{pj}\,x_p x_j
\;\le\;
\sum_{\substack{j\ne p\\ m_{pj}> 0}} m_{pj}\,x_p x_q
= r_p^+(\mathbf{M})\,x_p x_q,
\]
and, symmetrically,
\[
\sum_{j\ne p} m_{pj}\,x_p x_j
\;\ge\;
\sum_{\substack{j\ne p\\ m_{pj}< 0}} m_{pj}\,x_p x_j
\;\ge\;
-\sum_{\substack{j\ne p\\ m_{pj}< 0}} |m_{pj}|\,x_p x_q
= -\,r_p^-(\mathbf{M})\,x_p x_q.
\]
For the case $i=q$, we use that $x_j\le x_p$ for all $j\ne q$, then
\[
\sum_{j\ne q} m_{qj}\,x_q x_j
\le
\sum_{\substack{j\ne q\\ m_{qj}> 0}} m_{qj}\,x_q x_p
= r_q^+(\mathbf{M})\,x_p x_q,
\quad
\sum_{j\ne q} m_{qj}\,x_q x_j
\ge
\sum_{\substack{j\ne q\\ m_{qj}< 0}} m_{qj}\,x_q x_p=
-\,r_q^-(\mathbf{M})\,x_p x_q.
\]
\end{proof}

\begin{lemma}\label{lem:pol}
The quadratic functions $P^{\, \operatorname{low}}_{ij}$ and $P^{\, \operatorname{up}}_{ij}$ presented in Definition~\ref{def:bij} can  also be written as follows, for all $y\in\mathbb{R}$,
\begin{align}
P^{\, \operatorname{up}}_{ij}(y) &=
(y b_{ii}-a_{ii})(yb_{jj}-a_{jj})- (r_i^{+}(\mathbf{A})+y\,r_i^{-}(\mathbf{B})) \, (r_j^{+}(\mathbf{A})+y\,r_j^{-}(\mathbf{B})), \label{eq:puplong}\\
P^{\, \operatorname{low}}_{ij}(y) &=
(y b_{ii}-a_{ii})(y b_{jj}-a_{jj})- (r_i^{-}(\mathbf{A})+y\,r_i^{+}(\mathbf{B})) \, (r_j^{-}(\mathbf{A})+y\,r_j^{+}(\mathbf{B})). \label{eq:plowlong}
\end{align}
\end{lemma}
\begin{proof}
The definition of $P^{\, \operatorname{up}}_{i j}(y)$ can be retrieved by expanding the expression~\eqref{eq:puplong}:
\[
\begin{aligned}
&\bigl(b_{ii}b_{jj}-r_i^-(\mathbf B)r_j^-(\mathbf B)\bigr)\,y^2
+\bigl(a_{ii}a_{jj}-r_i^+(\mathbf A)r_j^+(\mathbf A)\bigr)\\
&\quad-\Bigl(a_{ii}b_{jj}+a_{jj}b_{ii}+r_i^+(\mathbf A)r_j^-(\mathbf B)+r_j^+(\mathbf A)r_i^-(\mathbf B)\Bigr)\,y\\
&=\rowmm{b}{i}{j}\,y^2
-s_{ij}^+({\A}, \B)\,y
+\rowpp{a}{i}{j}
= P^{\, \operatorname{up}}_{i j}(y)
\end{aligned} 
\]

Note that, due to symmetry in Definition~\ref{def:bij}, $P^{\, \operatorname{low}}_{i j}(y)$ in~\eqref{eq:plowlong} is obtained from $P^{\, \operatorname{up}}_{i j}(y)$ in~\eqref{eq:puplong} only by replacing the terms $r^+$ with $r^-$ and vice versa.

\end{proof}

\begin{lemma}\label{lem:root}
    Under the hypotheses of Theorem~\ref{thm:two-row-theorem}, the quadratic functions $P^{\, \operatorname{low}}_{ij}$ and $P^{\, \operatorname{up}}_{ij}$ are strictly convex and have real roots. As a consequence, 
    for any $y\in\R$, if $P^{\, \operatorname{up}}_{ij}(y)\le 0$,
    then $y\le \max\{y \in \R \mid P^{\, \operatorname{up}}_{ij}(y) =0\}$, or  if 
    $P^{\, \operatorname{low}}_{ij}(y)\le 0$,
    then $y\ge \min\{y \in \R \mid P^{\, \operatorname{low}}_{ij}(y) =0\}$.
    %is lower than the largest root of $P^{\, \operatorname{up}}_{ij}$ (and this root exists), and, symmetrically, any $\lambda\in\R$ satisfying $P^{\, \operatorname{low}}_{ij}(\lambda)\le 0$ is greater than the smallest root of $P^{\, \operatorname{low}}_{ij}$.
\end{lemma}

\begin{proof}
According to Lemma \ref{lemma:utility} $(ii)$, the leading coefficients $\rowmm{b}{i}{j}$ and $\rowpp{b}{i}{j}$ of the quadratic polynomials 
are positive, thus $P^{\, \operatorname{low}}_{ij}$ and $P^{\, \operatorname{up}}_{ij}$ are strictly convex. For a strictly convex quadratic polynomial $P$ with two real roots (not necessarily distinct), the interval in-between is precisely $\{y\in\R: P(y)\le 0\}$. 
Thus, it remains to show that $P^{\, \operatorname{low}}_{ij}$ and $P^{\, \operatorname{up}}_{ij}$ have nonnegative discriminants $\Delta^{\, \operatorname{low}}_{ij}$ and $\Delta^{\, \operatorname{up}}_{ij}$.
%, i.e., two real roots, and that the minimal root is nonnegative.

Since $\A$ is copositive and $\B$ is $\operatorname{PD}$, their diagonal terms are nonnegative, so are the one-row sums $r^+$ and $r^-$ from Definition~\ref{def:rpm}, hence:
%we have $a_{ii},a_{jj},b_{ii},b_{jj}\ge0$, and by definition we also have
%$r_i^\pm(\cdot),r_j^\pm(\cdot)\ge0$. Applying AM--GM inequality twice,
\[
a_{ii}b_{jj}+a_{jj}b_{ii}\ \ge\ 2\sqrt{a_{ii}a_{jj}b_{ii}b_{jj}},\quad
r_i^-(\A)r_j^+(\B)+r_j^-(\A)r_i^+(\B)\ \ge\ 2\sqrt{r_i^-(\A)r_j^-(\A)\,r_i^+(\B)r_j^+(\B)}.
\]
By summing and squaring each side, we get
\[
s_{ij}^-(\A,\B)^2
\ \ge\ \bigl(2\sqrt{a_{ii}a_{jj}b_{ii}b_{jj}}+2\sqrt{r_i^-(\A)r_j^-(\A)\,r_i^+(\B)r_j^+(\B)}\bigr)^2.
\]
Subtracting $4\, b_{ij}^+\,\rowmm{a}{i}{j}
=4\bigl(b_{ii}b_{jj}-r_i^+(\B)r_j^+(\B)\bigr)\bigl(a_{ii}a_{jj}-r_i^-(\A)r_j^-(\A)\bigr)$ yields
\[
\Delta^{\,\operatorname{low}} _{ij}\ := s_{ij}^-(\A,\B)^2- \, 4\, b_{ij}^+\,\rowmm{a}{i}{j} \ \ge\ 4\Bigl(\sqrt{a_{ii}a_{jj}\,r_i^+(\B)r_j^+(\B)}+\sqrt{b_{ii}b_{jj}\,r_i^-(\A)r_j^-(\A)}\Bigr)^2\ \ge\ 0.
\]
The proof is symmetric for $\Delta^{\, \operatorname{up}}_{ij}$
by inverting $r^+$ and $r^-$ as observed in Lemma~\ref{lem:pol}, i.e.,
\[
\Delta^{\,\operatorname{up}} _{ij}\ := s_{ij}^+(\A,\B)^2- \, 4\, b_{ij}^-\,\rowpp{a}{i}{j} \ \ge\ 4\Bigl(\sqrt{a_{ii}a_{jj}\,r_i^-(\B)r_j^-(\B)}+\sqrt{b_{ii}b_{jj}\,r_i^+(\A)r_j^+(\A)}\Bigr)^2\ \ge\ 0.
\]

%By using analogous proof, we can also prove 
% for $\Delta^{\,\operatorname{low}} _{ij}$ is . AM--GM inequality gives
%\[
%\Delta^{\,\operatorname{low}} _{ij} (\A,\B)\ \ge\ 4\Bigl(\sqrt{a_{ii}a_{jj}\,r_i^+(\B)r_j^+(\B)}+\sqrt{b_{ii}b_{jj}\,r_i^-(\A)r_j^-(\A)}\Bigr)^2\ \ge\ 0.
%\]
\end{proof}

%In Lemma \ref{lemma: discriminants}, we will prove that these quadratic polynomials have nonnegative discriminants, then have real roots. These properties are exploited to show that a complementary eigenvalue $\lambda$ can be located in an interval defined by the smallest root of $P^{\, \operatorname{low}}_{ij}$ and the largest root of $P^{\, \operatorname{up}}_{ij}$ for any $i \neq j$.

% We also prove that any $\lambda\in\R$ satisfying $P^{\, \operatorname{up}}_{ij}(\lambda)\le 0$ is lower than the largest root of $P^{\, \operatorname{up}}_{ij}$ and any satisfying $P^{\, \operatorname{low}}_{ij}(\lambda)\le 0$ is greater than the smallest root of $P^{\, \operatorname{low}}_{ij}$.

\begin{proof}[Proof of Theorem~\ref{thm:two-row-theorem}]
Let $(\x,\lambda)$ be a complementarity eigenpair of $(\A,\B)$. The proof is split into two cases, depending on whether the vector $\x$ has one nonzero value (i.e., $\x=\e_p$ for some index $p\in[n]$) or at least two.\\

%. In Case~1, we consider a solution $(\lambda,\x)$ with $\x=\mathbf{e}_p$, for some $p\in[n]$, and show that $\lambda=a_{pp}/b_{pp}$ implies $P^{\, \operatorname{low}}_{pq}(\lambda)\le 0$ and $P^{\, \operatorname{up}}_{pq}(\lambda)\le 0$ for any $q\in [n]\setminus\{p\}$. Thus, it belongs to at least {one of the} $n-1$ intervals defining $K_2$.
%\medskip
%In Case~2, we have \(\x \neq \mathbf{e}_p\). Let \(p,q\in[n]\) be such that \(x_p=\max_i x_i\) and \(x_q=\max_{i\neq p}x_i\), hence \(x_p\ge x_q>0\). Using the \(p\)- and \(q\)-row relations in the definition of an EiCP~\eqref{eq:EiCP}, 
% in the complementarity system \((\A-\lambda \B)\x\ge \mathbf0\) together with \(\x^\top(\A-\lambda \B)\x=0\), 
%we can derive two inequalities which, after rearrangement, are captured by  \(P^{\mathrm{up}}_{pq}\) and \(P^{\mathrm{low}}_{pq}\) defined above. 

\noindent\textbf{Case 1: $\x=\mathbf{e}_p$, $p\in[n]$.}
%Assume that $\mathbf{e}_p$ is a complementarity eigenvector for some index $p\in[n]$. Therefore, 
Due to the complementarity constraints, the $p$-row of \eqref{eq:EiCP_a} reads as $\lambda b_{pp}-a_{pp}=0$, that is, $\lambda={a_{pp}}/{b_{pp}}$ which is nonnegative due to the positivity assumptions on $\A$ and $\B$.
%, we have $a_{pp}\geq 0$ and $b_{pp}>0$, hence $\lambda \geq 0$. 
%We show that for any other index $j\in[n]$, with $j\neq p$, we have $P^{\, \operatorname{up}}_{pj}(\lambda)\le 0$ and $P^{\, \operatorname{low}}_{pj}(\lambda)\le 0$ with $\lambda={a_{pp}}/{b_{pp}}$. At this aim, let us 
For any index $j\in[n]$, with $j\neq p$, according to Lemma~\ref{lem:pol}, $P^{\, \operatorname{up}}_{p j}$ and $P^{\, \operatorname{low}}_{p j}$  evaluate at $\lambda = {a_{pp}}/{b_{pp}}$ as 
\[
\begin{aligned}
P^{\, \operatorname{up}}_{p j}(\lambda) &= -\Bigl(r_p^+(\mathbf A)+\lambda\,r_p^-(\mathbf B)\Bigr)\Bigl(r_j^+(\mathbf A)+\lambda\,r_j^-(\mathbf B)\Bigr)\\
P^{\, \operatorname{low}}_{p j}(\lambda) &= -\Bigl(r_p^-(\mathbf A)+\lambda\,r_p^+(\mathbf B)\Bigr)\Bigl(r_j^-(\mathbf A)+\lambda\,r_j^+(\mathbf B)\Bigr),
\end{aligned}
\]
%By using the definitions in~\eqref{def:bij}, this gives
%\[
%\begin{aligned}
%b_{pp}^2\,P^{\, \operatorname{up}}_{p j}(a_{pp}/b_{pp})
%&=\rowmm{b}{p}{j}\,a_{pp}^2
%-s_{pj}^+({\A}, \B)\,a_{pp}b_{pp}
%+\rowpp{a}{p}{j}b_{pp}^2\\
%&=\bigl(b_{pp}b_{jj}-r_p^-(\mathbf B)r_j^-(\mathbf B)\bigr)\,a_{pp}^2
%+\bigl(a_{pp}a_{jj}-r_p^+(\mathbf A)r_j^+(\mathbf A)\bigr)\,b_{pp}^2\\
%&\quad-\Bigl(a_{pp}b_{jj}+a_{jj}b_{pp}+r_p^+(\mathbf A)r_j^-(\mathbf B)+r_j^+(\mathbf A)r_p^-(\mathbf B)\Bigr)\,a_{pp}b_{pp}\\
%&= -r_p^-(\mathbf B)r_j^-(\mathbf B)a_{pp}^2
%- r_p^+(\mathbf A)r_j^+(\mathbf A)b_{pp}^2
%- \bigl(r_p^+(\mathbf A)r_j^-(\mathbf B)+r_j^+(\mathbf A)r_p^-(\mathbf B)\bigr)a_{pp}b_{pp} \\
%& = -\,b_{pp}^2\Bigl(r_p^+(\mathbf A)+\frac{a_{pp}}{b_{pp}}\,r_p^-(\mathbf B)\Bigr)\Bigl(r_j^+(\mathbf A)+\frac{a_{pp}}{b_{pp}}\,r_j^-(\mathbf B)\Bigr).
%\end{aligned}
%\]
% This simplifies by cancelling pairs $a_{pp}^2 b_{pp}b_{jj}$ and $a_{pp}a_{jj}b_{pp}^2$:
% \[
% \begin{aligned}
% b_{pp}^2\,P^{\, \operatorname{up}}_{p j}(a_{pp}/b_{pp})
% &= -\left[ a_{pp}^2\,r_p^-(\mathbf B)r_j^-(\mathbf B)
% + a_{pp}b_{pp}\,\bigl(r_p^+(\mathbf A)r_j^-(\mathbf B)+r_j^+(\mathbf A)r_p^-(\mathbf B)\bigr) + b_{pp}^2\,r_p^+(\mathbf A)r_j^+(\mathbf A)\right].\\
% &=
% -\,b_{pp}^2\Bigl(r_p^+(\mathbf A)+\frac{a_{pp}}{b_{pp}}\,r_p^-(\mathbf B)\Bigr)\Bigl(r_j^+(\mathbf A)+\frac{a_{pp}}{b_{pp}}\,r_j^-(\mathbf B)\Bigr).
% \end{aligned}
% \]
which are both nonpositive, since $\lambda$, $r^+$, and $r^-$ are all nonnegative.
% Then, for $\lambda = a_{pp}/b_{pp} \ge 0$, and since row sums $r_i^\pm$ are nonnegative, one has
%
%\medskip
%  regarding the lower polynomial $P^{\, \operatorname{low}}_{p j}$ at $\lambda=a_{pp}/b_{pp}$, one has
% \[
% \begin{aligned}
% b_{pp}^2\,P^{\, \operatorname{low}}_{p j}(\lambda)
% &=\rowpp{b}{p}{j}\,a_{pp}^2
% -s_{pj}^-({\A}, \B)\,a_{pp}b_{pp}
% +\rowmm{a}{p}{j}\\
% &=\bigl(b_{pp}b_{jj}-r_p^+(\mathbf B)r_j^+(\mathbf B)\bigr)\,a_{pp}^2
% +\bigl(a_{pp}a_{jj}+r_p^-(\mathbf A)r_j^-(\mathbf A)\bigr)\,b_{pp}^2\\
% &\quad-\bigl(a_{pp}b_{jj}+a_{jj}b_{pp}+r_p^-(\mathbf A)r_j^+(\mathbf B)+r_j^-(\mathbf A)r_p^+(\mathbf B)\bigr)\,a_{pp}b_{pp}\\
% &= -\,b_{pp}^2\Bigl(r_p^-(\mathbf A)+\frac{a_{pp}}{b_{pp}}\,r_p^+(\mathbf B)\Bigr)\Bigl(r_j^-(\mathbf A)+\frac{a_{pp}}{b_{pp}}\,r_j^+(\mathbf B)\Bigr) \le 0.
% \end{aligned}
% \]
%As in the proof of Lemma~\ref{lemma: discriminants}, we derive by symmetry that
% \[
% \begin{aligned}
% P^{\, \operatorname{low}}_{p j}(\lambda) = -\Bigl(r_p^-(\mathbf A)+\frac{a_{pp}}{b_{pp}}\,r_p^+(\mathbf B)\Bigr)\Bigl(r_j^-(\mathbf A)+\frac{a_{pp}}{b_{pp}}\,r_j^+(\mathbf B)\Bigr) \le 0.
% \end{aligned}
% \]
%with $\lambda =a_{pp}/b_{pp}$.
By Lemma~\ref{lem:root}, $\lambda$ belongs to at least $n-1$ intervals (with $i=p$) defining $K_2$.\\

\noindent\textbf{Case 2:} $\mathbf{x}_p \geq \mathbf{x}_q>0\mathbf{,} \ p\neq q$\textbf{.} 
%\bigskip
%Then $(x_p x_q)^2>0$.
From now on, assume that $\x$ has at least two positive elements and $p$ and $q$ are distinct indices in $[n]$ of the largest and second largest elements of $\x$ as in Lemma~\ref{lem:two-rows-bounds}. The complementarity condition~\eqref{eq:EiCP_c} reads as
\begin{equation}\label{eq:two-rows-detailed}
%\begin{aligned}
(\lambda b_{ii}-a_{ii})\,x_i^2=\sum_{j\ne i}(a_{ij}-\lambda b_{ij})\,x_i x_j, \text{ for } i\in [n].
%\end{aligned}
\end{equation}
%\begin{equation}\label{eq:two-rows-detailedq}
%\begin{aligned}
%(\lambda b_{qq}-a_{qq})\,x_q^2=\sum_{j\ne q}(a_{qj}-\lambda b_{qj})\,x_q x_j,  
%\end{aligned}
%\end{equation}
Applying bounds from Lemma~\ref{lem:two-rows-bounds} to this identity for $i=p$ or $i=q$, %into the identities in \eqref{eq:two-rows-detailed} and \eqref{eq:two-rows-detailedq} 
and using $\lambda\ge 0$ by Property~\ref{prop:nonnegla} yield
\begin{equation}\label{eq:row-upper}
(\lambda b_{ii}-a_{ii})\,x_i^2
\;\le\; \bigl(r_i^+(\mathbf A)+\lambda r_i^-(\mathbf B)\bigr)\,x_p x_q, \text{ for } i\in\{p,q\}.
\end{equation}

If $\lambda>a_{ii}/b_{ii}$ for both $i\in\{p,q\}$, then left-hand sides of inequalities~\eqref{eq:row-upper} are positive. Multiplying the two inequalities and dividing by $x_p^2x_q^2>0$ give
\[
(\lambda b_{pp}-a_{pp})(\lambda b_{qq}-a_{qq})\ \le\ (r_p^{+}(\mathbf{A})+\lambda\,r_p^{-}(\mathbf{B})) \, (r_q^{+}(\mathbf{A})+\lambda\,r_q^{-}(\mathbf{B})).
\]
By rearranging all the terms to the left side and by using~\eqref{eq:puplong}, we have $P^{\mathrm{up}}_{pq}(\lambda)\le 0$.
% We recognize  from Lemma~\ref{lem:pol}, then $\lambda\le \max\{y \in \R \mid P^{\, \operatorname{up}}_{pq}(y) =0\}$ from Lemma~\ref{lem:root}.

We consider now that $\lambda\le a_{ii}/b_{ii}$ for some $i\in\{p,q\}$. From Case 1, we have that $P^{\mathrm{up}}_{pq}(a_{ii}/b_{ii})\le 0$, which allows concluding that 
 $\lambda\le a_{ii}/b_{ii} \leq \max\{y \in \R \mid P^{\, \operatorname{up}}_{pq}(y) =0\}$.
 
% Otherwise, if $\lambda\le a_{ii}/b_{ii}$ for some $i\in\{p,q\}$, then, as in Case 1, we prove that $P^{\mathrm{up}}_{pq}(a_{ii}/b_{ii})\le 0$, yielding
%  $\lambda\le a_{ii}/b_{ii} \leq \max\{y \in \R \mid P^{\, \operatorname{up}}_{pq}(y) =0\}$. %Thus, by transitivity, $\lambda\le a_{ii}/b_{ii}\le \max\{y \in \R \mid P^{\, \operatorname{up}}_{pq}(y) =0\}$.
%\end{itemize}
%\bigskip
%\vs{VS: I still have the following thought: from the above reasoning, we have $\min\{y \in \R \mid P^{\, \operatorname{up}}_{pq}(y) =0\} \leq \lambda\le \max\{y \in \R \mid P^{\, \operatorname{up}}_{pq}(y) =0\}$, don't we? Why we do not exploit that? Similarly, in item (b) } \as{AS: I am not sure. Because in the second "if" (when $\lambda \le a_{pp}/b_{pp}$ or $\lambda\le a_{qq}/b_{qq}$) we have no guarantee that $\lambda \ge \min\{y \in \R \mid P^{\, \operatorname{up}}_{pq}(y) =0\}$. And the same situation occurs symmetrically in case (b).} \vs{I am not convinced :-)}
%\bigskip

Symmetric arguments apply to prove that $\lambda\ge \min\{y \in \R \mid P^{\, \operatorname{low}}_{pq}(y) =0\}$. We
first multiply both sides of identity~\eqref{eq:two-rows-detailed} by $-1$, then we apply Lemma~\ref{lem:two-rows-bounds} with $\lambda\ge 0$, yielding 
\begin{equation}\label{eq:row-ower}
(a_{ii}-\lambda b_{ii})\,x_i^2
\;\le\; \bigl(\lambda r_i^+(\mathbf B) + r_i^-(\mathbf A)\bigr)\,x_p x_q, \text{ for } i\in\{p,q\}.
\end{equation}
If $\lambda < a_{ii}/b_{ii}$ in~\eqref{eq:row-ower} for both $i\in\{p,q\}$ then these inequalities have positive left-hand sides; their product reads $P^{\mathrm{low}}_{pq}(\lambda)\le 0$, and then $\lambda\ge \min\{y \in \R \mid P^{\, \operatorname{low}}_{pq}(y) =0\}$. If $\lambda\ge a_{ii}/b_{ii}$ for some $i\in\{p,q\}$, then $P^{\mathrm{low}}_{pq}(a_{ii}/b_{ii})\le 0$, and, by transitivity, $\lambda\ge a_{ii}/b_{ii}\ge \min\{y \in \R \mid P^{\, \operatorname{low}}_{pq}(y) =0\}$.
%
% With similar arguments, from \eqref{eq:p-row-lower} and \eqref{eq:q-row-lower}, we have
%\begin{equation*}
%\begin{aligned}
%(a_{pp}-\lambda b_{pp})x_p^2\ \le\ (\lambda\,r_p^{+}(\mathbf{B})+r_p^{-}(\mathbf{A}))\,x_p x_q,\\
%(a_{qq}-\lambda b_{qq})x_q^2\ \le\ (\lambda\,r_q^{+}(\mathbf{B})+r_q^{-}(\mathbf{A}))\,x_p x_q.
%\end{aligned}
%\label{eq:lower-pq}
%\end{equation*}
%\medskip 
%\begin{itemize}
%    \item If $\lambda<a_{pp}/b_{pp}$ and $\lambda<a_{qq}/b_{qq}$, both left-hand sides in the above inequalities are strictly positive. Thus, 
%\begin{equation*}
%  (a_{pp}-\lambda b_{pp})(a_{qq}-\lambda b_{qq}) \ \le\ (\lambda\,r_p^{+}(\mathbf{B})+r_p^{-}(\mathbf{A})) \, (\lambda\,r_p^{+}(\mathbf{B})+r_p^{-}(\mathbf{A})).  
%\end{equation*}
%
%Rearranging gives $P^{\mathrm{low}}_{pq}(\lambda)\le 0$, and then $\lambda\ge \min\{y \in \R \mid P^{\, \operatorname{low}}_{pq}(y) =0\}$.
%\medskip
%    \item If $\lambda\ge a_{pp}/b_{pp}$ or $\lambda\ge a_{qq}/b_{qq}$.
%With $t_p,t_q$ as before, Case~1 gives $P^{\mathrm{low}}_{pq}(t_p)\le 0$ and $P^{\mathrm{low}}_{pq}(t_q)\le 0$.
%Hence $\lambda\ge \min\{y \in \R \mid P^{\, \operatorname{low}}_{pq}(y) =0\}$.
%\end{itemize}
%\end{enumerate}

%\bigskip
Thus we have shown that $\lambda$ belongs to at least one interval, with $(i,j)=(p,q)$, of $K_2$.
%We have shown that, for the same pair $(p,q)$ associated with a complementarity eigenvector $\mathbf{x}$, we have $\lambda\le \max\{y \in \R \mid P^{\, \operatorname{up}}_{pq}(y) =0\}$ and $\lambda\ge \min\{y \in \R \mid P^{\, \operatorname{low}}_{pq}(y) =0\}$.
%Taking the union over all ordered pairs $i\ne j$ yields $\lambda\in K_2$.

\end{proof}

\begin{remark}
From Case 1 in the latter proof, we see that each interval in the union defining $K_2$ is well-defined for every pair of indices \(i,j \in [n]\), as it contains the values $a_{ii}/b_{ii}$ and $a_{jj}/b_{jj}$. Indeed, 
%[The set $K_2$ is a union of well-defined intervals]
%We can show that 
%$$ \{y \in \R \mid P^{\, \operatorname{low}}_{ij}(y) \leq  0\} \cap \{y \in \R \mid P^{\, \operatorname{up}}_{ij}(y) \leq  0\} \neq \emptyset $$
applying Lemma~\ref{lem:pol} at $y=a_{ii}/b_{ii}$, which is nonnegative by hypothesis, yields $P^{\, \operatorname{low}}_{ij}(a_{ii}/b_{ii})\le 0$ and $P^{\, \operatorname{up}}_{ij}(a_{ii}/b_{ii})\le 0$, then, from Lemma~\ref{lem:root}, 
\[
\min\{y \in \R \mid P^{\, \operatorname{low}}_{ij}(y) = 0\} \le a_{ii}/b_{ii} \le \max\{y \in \R \mid P^{\, \operatorname{up}}_{ij}(y) = 0\}. 
\] 

%. At this aim, it is sufficient to consider $y = a_{ii}/b_{ii}  \geq 0$ and evaluate $P^{\, \operatorname{low}}_{ij}$ and $P^{\, \operatorname{up}}_{ij}$ at this value. Similarly to what shown in Case~1 above, we can see that both quadratic functions evaluated at $y = a_{ii}/b_{ii}$ are nonpositive. With this in mind,  we can state that for every pair of indices \(i,j \in [n]\), such that \(i \neq j\), there exists a $t \geq 0$ belonging to the above intersection such that
%\[
%\min\{y \in \R \mid P^{\, \operatorname{low}}_{ij}(y) = 0\} \le t \le \max\{y \in \R \mid P^{\, \operatorname{up}}_{ij}(y) = 0\}.
%\]
\end{remark}

\section{Comparing the localization sets}\label{sec3}
%Comparison between sets $K_1$ and $K_2$}

In this section, we show that the two-row set $K_2$ presented in Theorem~\ref{thm:two-row-theorem} is smaller than the one-row set $K_1$ given in Theorem~\ref{thm:one-row-theorem}. At this aim, it is necessary to consider the same assumptions for both statements (in particular, the copositivity of $\A$ in $K_1$). Note that $K'_1 \subset K_1$ (by construction) and then prove that $K_2 \subset K'_1$ directly implies $K_2 \subset K_1$. Therefore, we introduce below the notations $C^{\operatorname{up}}_{ij}$ and $C^{\operatorname{low}}_{ij}$ representing the two-row version of the endpoints from the intervals, which define $K'_1$ in Corollary~\ref{cor:onecop}.

% In addition, if we allow $i=j$ under the same statement of Theorem~\ref{thm:two-row-theorem} the set $K_2$ would reduce to $K_1$.

% \begin{lemma}\label{lemma:derivative-both}
% Let $\A,\B\in\mathbb S_n$. Assume that $\B$ is positive definite and strictly diagonally dominant, and $\A$ is copositive.
% For any $i,j \in [n]$, such that $i\ne j$, define
% \begin{equation}\label{eq:Cij}
% C^{\operatorname{up}}_{ij}
% := \max\bigg\{\frac{a_i^{+}}{b_i^{-}},\ \frac{a_j^{+}}{b_j^{-}}\bigg\},\quad
% C^{\operatorname{low}}_{ij}
% := \min\bigg\{\frac{a_i^{-}}{b_i^{+}},\ \frac{a_j^{-}}{b_j^{+}}\bigg\}, 
% \end{equation}
% \begin{equation}\label{eq:ystar}
% y^{\star,\operatorname{up}}_{ij}
% := \frac{s_{ij}^{+}(\A,\B)}{2\,b_{ij}^{-}},\quad
% y^{\star,\operatorname{low}}_{ij}
% := \frac{s_{ij}^{-}(\A,\B)}{2\,b_{ij}^{+}}.   
% \end{equation}

% Thus, note that $y^{\star,\operatorname{up}}_{ij}$ is the vertex of $P^{\operatorname{up}}_{ij}$ and $y^{\star,\operatorname{low}}_{ij}$ is the vertex of $P^{\operatorname{low}}_{ij}$, and we have $C^{\operatorname{up}}_{ij}\ge y^{\star,\operatorname{up}}_{ij}$ and $C^{\operatorname{low}}_{ij}\le y^{\star,\operatorname{low}}_{ij}$ for every $i,j \in [n]$, such that $i\ne j$.
% \end{lemma}

\begin{lemma}\label{lemma:derivative-both}
Let $\A,\B\in\mathbb S_n$. Assume that $\B$ is positive definite and strictly diagonally dominant, and $\A$ is copositive.
For any $i,j \in [n]$, such that $i\ne j$, let us define
\begin{equation}\label{eq:Cij}
C^{\operatorname{up}}_{ij}
:= \max\bigg\{\frac{a_i^{+}}{b_i^{-}},\ \frac{a_j^{+}}{b_j^{-}}\bigg\},\quad
C^{\operatorname{low}}_{ij}
:= \min\bigg\{\frac{a_i^{-}}{b_i^{+}},\ \frac{a_j^{-}}{b_j^{+}}\bigg\}, 
\end{equation}
\begin{equation}\label{eq:ystar}
y^{\star,\operatorname{up}}_{ij}
:= \frac{s_{ij}^{+}(\A,\B)}{2\,b_{ij}^{-}},\quad
y^{\star,\operatorname{low}}_{ij}
:= \frac{s_{ij}^{-}(\A,\B)}{2\,b_{ij}^{+}}.   
\end{equation}

Then, the following properties hold
\begin{enumerate}[(a)]
    \item $y^{\star,\operatorname{up}}_{ij}$ and $y^{\star,\operatorname{low}}_{ij}$ are the vertices of $P^{\operatorname{up}}_{ij}$ and $P^{\operatorname{low}}_{ij}$, respectively.
    \item $y^{\star,\operatorname{up}}_{ij} \le C^{\operatorname{up}}_{ij} $ and $y^{\star,\operatorname{low}}_{ij} \ge C^{\operatorname{low}}_{ij}$ for every $i,j \in [n]$, such that $i\ne j$.
\end{enumerate}
\end{lemma}

\begin{proof}
%Since $\B$ is $\operatorname{SDD}$, we have $b_i^\pm = b_{ii}\pm r_i^\pm(\B)>0$ and $b_j ^\pm= b_{jj}\pm r_j^\pm(\B)>0$.
To prove the statement in (a), recall from Lemma~\ref{lem:root} that $P^{\,\operatorname{up}}_{ij}$ is a strictly convex quadratic function. Its minimum occurs at the vertex of this parabola, which can be found by setting the first derivative equal to zero. This gives $y^{\star,\,\operatorname{up}}_{ij}$ in~\eqref{eq:ystar}. Similarly, $y^{\star,\,\operatorname{low}}$ in~\eqref{eq:ystar} is the vertex of 
 $P^{\operatorname{low}}_{ij}$. 
% $(P^{\,\operatorname{up}}_{ij})'(y^{\star,\,\operatorname{up}}_{ij})=2b_{ij}^{-}y^{\star,\,\operatorname{up}}-s_{ij}^{+}(\A,\B)=0$ 
% then $y^{\star,\,\operatorname{up}}$ is the vertex of $P^{\,\operatorname{up}}_{ij}$, 
% and similarly, $y^{\star,\,\operatorname{low}}$ 
% is the vertex of 
% $P^{\operatorname{low}}_{ij}$. 
% Following the same reasoning, in (i) we prove the upper inequality and in (ii) the lower one.

The two inequalities in (b) can be obtained by the following reasoning.
\begin{enumerate}[(i)]
    \item From the definition of $C^{\operatorname{up}}_{ij}$ and  $b_i^->0$ and $b_j^->0$ as $\B$ is PD and SDD:
\[
a_i^{+}\le C^{\operatorname{up}}_{ij}\,b_i^{-},\quad
a_j^{+}\le C^{\operatorname{up}}_{ij}\,b_j^{-}.
\]
Multiplying the first inequality by $b_{jj}+r_j^{-}(\B)$ and the second one by $b_{ii}+r_i^{-}(\B)$ and adding the two resulting inequalities yield
\[
a_i^{+}(b_{jj}+r_j^{-}(\B))+a_j^{+}(b_{ii}+r_i^{-}(\B))
\ \le\
C^{\operatorname{up}}_{ij}\,\bigl(b_i^{-}(b_{jj}+r_j^{-}(\B))+b_j^{-}(b_{ii}+r_i^{-}(\B))\bigr).
\]
The right-hand side equals $2\,C^{\operatorname{up}}_{ij}b_{ij}^{-}$ and expanding the left-hand side gives
\[
s_{ij}^{+}(\A,\B)\;+\;a_{ii}r_j^{-}(\B)+a_{jj}r_i^{-}(\B)+r_i^{+}(\A)b_{jj}+r_j^{+}(\A)b_{ii},
\]
where all terms are nonnegative. Thus $s_{ij}^{+}(\A,\B)\le 2\,C^{\operatorname{up}}_{ij}\,b_{ij}^{-}$, which yields
$$y^{\star,\operatorname{up}}_{ij}:= \frac{s_{ij}^{+}(\A,\B)}{2\,b_{ij}^{-}} \le C^{\operatorname{up}}_{ij}.$$

\item Similarly, as $b^- > 0$ we have
%From the definition of $C^{\operatorname{low}}_{ij}$ we have
%\[
%a_i^{-}\ge C^{\operatorname{low}}_{ij}\,b_i^{+},\quad
%a_j^{-}\ge C^{\operatorname{low}}_{ij}\,b_j^{+}.
%\]
%Multiply the first inequality by $b_{jj}-r_j^{+}(\B)$ and the second by $b_{ii}-r_i^{+}(\B)$ and add:
\[
a_i^{-}(b_{jj}-r_j^{+}(\B))+a_j^{-}(b_{ii}-r_i^{+}(\B))
\ \ge\
C^{\operatorname{low}}_{ij}\,\bigl(b_i^{+}(b_{jj}-r_j^{+}(\B))+b_j^{+}(b_{ii}-r_i^{+}(\B))\bigr)
\]
%The bracket equals $2\,b_{ij}^{+}$. Repeating the same yields
which yields
$$s_{ij}^{-}(\A,\B)\ge s_{ij}^{-}(\A,\B)\;-\;a_{ii}r_j^{+}(\B)-a_{jj}r_i^{+}(\B)-r_i^{-}(\A)b_{jj}-r_j^{-}(\A)b_{ii}\ge 2\,C^{\operatorname{low}}_{ij}\,b_{ij}^{+},$$ hence
$y^{\star,\operatorname{low}}_{ij} \ge C^{\operatorname{low}}_{ij}$.
\end{enumerate}
\end{proof}
%\vs{We do not use the fact that $y^{\star,\,\operatorname{up}}_{ij}$ and $y^{\star,\,\operatorname{low}}_{ij}$ are the vertex of the parabolas in the Lemma above, right?} \as{No, we just introduce this information to contextualize the next theorem.} \sd{yes we do in the proof below}

In the proof of the following theorem, we can assume without loss of generality that $C^{\operatorname{low}}_{ij} > 0$ because, otherwise, from the definition of  $K'_1$ in Corollary~\ref{cor:onecop}, the lower bound of the interval associated with $(i,j)$ is zero (which is always equal to or less than any endpoint of the intervals that define $K_2$).

\begin{theorem}[$K_2$ is tighter than $K_1$]\label{thm:2row-in-1row}
Let $\A,\B\in\mathbb S_n$. Assume that $\B$ is positive definite and strictly diagonally dominant, and that $\A$ is copositive, then $K_2 \subset K_1$ (precisely $K_2 \subset K'_1 \subset K_1$).
\end{theorem}

\begin{proof}
For any $i\in[n]$ and $y\in\mathbb R$, let us denote 
$$X_i(y)=b_{ii}y-a_{ii},\quad 
U_i(y)=r_i^+(\A)+yr_i^-(\B),\quad 
L_i(y)=r_i^-(\A)+yr_i^+(\B).$$
By Lemma~\ref{lem:pol}, for $i,j\in[n], i\neq j$, we have
\begin{equation} \label{eq:PXUL}
P^{\,\operatorname{up}}_{ij}(y) =X_i(y) \, X_j(y)-U_i(y) \, U_j(y),\quad 
P^{\,\operatorname{low}}_{ij}(y) =X_i(y) \, X_j(y)-L_i(y) \, L_j(y).   
\end{equation}
%recall the two-row polynomials $P^{\,\operatorname{up}}_{ij}$ and $P^{\,\operatorname{low}}_{ij}$ from %Theorem~\ref{thm:two-row-theorem}. By algebraic manipulation it can be seen that they admit the following factorizations
%\begin{equation}\label{eq:Pplus-fact}
%P^{\,\operatorname{up}}_{ij}(y)
%=\underbrace{(b_{ii}y-a_{ii})(b_{jj}y-a_{jj})}_{=:~X_i(y)\,X_j(y)}
%-\underbrace{\bigl(r_i^{+}(\A)+y\,r_i^{-}(\B)\bigr)\bigl(r_j^{+}(\A)+y\,r_j^{-}(\B)\bigr)}_{=:~U_i(y)\,U_j(y)},
%\end{equation}
%\begin{equation}\label{eq:Pminus-fact}
%P^{\,\operatorname{low}}_{ij}(y)
%=\underbrace{(a_{ii}-b_{ii}y)(a_{jj}-b_{jj}y)}_{=:~Y_i(y)\,Y_j(y)}
%-\underbrace{\bigl(r_i^{-}(\A)+y\,r_i^{+}(\B)\bigr)\bigl(r_j^{-}(\A)+y\,r_j^{+}(\B)\bigr)}_{=:~L_i(y)\,L_j(y)}.
%\end{equation}
Furthermore, $U_i(y)\ge 0$ and $L_i(y)\ge 0$ for $y\ge 0$ and we have the equivalences:
\begin{align}
X_i(y)>U_i(y)
~&\Longleftrightarrow~ (b_{ii}-r_i^{-}(\B))\,y>a_{ii}+r_i^{+}(\A)
~\Longleftrightarrow~ y> {a_i^{+}}/{\,b_i^{-}\,}, \label{eq:XiUi}\\
-X_i(y)>L_i(y)
~&\Longleftrightarrow~ (b_{ii}+r_i^{+}(\B))\,y<a_{ii}-r_i^{-}(\A)
~\Longleftrightarrow~ y< {a_i^{-}}/{\,b_i^{+}\,}. \label{eq:YiLi}
\end{align}

\begin{itemize}
\item If $y>\max\!\bigg\{\frac{a_i^{+}}{b_i^{-}},\,\frac{a_j^{+}}{b_j^{-}}  \bigg\} = C^{\,\operatorname{up}}_{ij}$, then $X_i(y)>U_i(y)\ge 0$ and $X_j(y)>U_j(y)\ge 0$. %and by \eqref{eq:Pplus-fact}.
By using the definition of $P^{\,\operatorname{up}}_{ij}$ in~\eqref{eq:PXUL}, these conditions yield to $P^{\,\operatorname{up}}_{ij}(y)> 0$. Since $P^{\,\operatorname{up}}_{ij}$ is a strictly convex quadratic function, the values of $y$ such that $P^{\,\operatorname{up}}_{ij}$ is positive lie beyond its smallest and largest root. However, from Lemma~\ref{lemma:derivative-both}, we get $y> C^{\,\operatorname{up}}_{ij} \geq y^{\star,\,\operatorname{up}}_{ij}$, that is $y$ lies to the \emph{right} of the largest root of $P^{\,\operatorname{up}}_{ij}$.

%\[
%P^{\,\operatorname{up}}_{ij}(y)=X_i(y) \, X_j(y)-U_i(y) \, U_j(y)>0.
%\]
% As in Lemma~\ref{lem:root}, the strict convexity of the quadratic function $P^{\,\operatorname{up}}_{ij}$ implies that $y$ lies beyond its smallest and largest root. From Lemma~\ref{lemma:derivative-both}, $y^{\star,\,\operatorname{up}}_{ij}\le C^{\,\operatorname{up}}_{ij}<y$ and it is the vertex of $P^{\,\operatorname{up}}_{ij}$, then 
%Since $P^{\,\operatorname{up}}_{ij}$ is strictly convex (its leading coefficient $b_{ij}^{-}>0$) with vertex $y^{\star,\,\operatorname{up}}_{ij} \le \max\!\bigg\{\frac{a_i^{+}}{b_i^{-}},\,\frac{a_j^{+}}{b_j^{-}}\bigg\}$ by Lemma~\ref{lemma:derivative-both}, such a 
%$y$ lies to the \emph{right} of the largest root of $P^{\,\operatorname{up}}_{ij}$.
\smallskip

    \item If $0 \le y<\min\!\bigg\{\frac{a_i^{-}}{b_i^{+}},\,\frac{a_j^{-}}{b_j^{+}}\bigg\} = C^{\operatorname{low}}_{ij}$, then $-X_i(y)>L_i(y)\ge 0$ and $-X_j(y)>L_j(y)\ge 0$, hence $P^{\,\operatorname{low}}_{ij}(y)>0$.
%\vs{what if $a^-$ is negative?} \as{If $a^{-}$ is negative the bound from $K_1 ^\prime$ is zero.}
    %and by \eqref{eq:Pminus-fact}
%\[
%P^{\,\operatorname{low}}_{ij}(y)=Y_i(y) \, Y_j(y)-L_i(y) \, L_j(y)>0.
%\]
Again, from 
%Lemma~\ref{lem:root} and 
Lemma~\ref{lemma:derivative-both},
%$P^{\,\operatorname{low}}_{ij}$ is strictly convex (leading coefficient $b_{ij}^{+}>0$) with vertex $y^{\star,\,\operatorname{low}}_{ij} \ge \min\!\bigg\{\frac{a_i^{-}}{b_i^{+}},\,\frac{a_j^{-}}{b_j^{+}}\bigg\}$ by Lemma~\ref{lemma:derivative-both}, so such a 
$y$ lies to the \emph{left} of the vertex $y^{\star,\,\operatorname{low}}_{ij}$ of $P^{\,\operatorname{low}}_{ij}$, and to its smallest root.
\end{itemize}
\smallskip
Since the endpoints that define $ K'_1$ are $\frac{a_i^{-}}{b_i^{+}}$ and $\frac{a_i^{+}}{b_i^{-}}$, for $i\in[n]$, 
and the endpoints that define $K_2$ are the smallest roots of $P^{\,\operatorname{low}}_{ij}$ and largest roots of $P^{\,\operatorname{up}}_{ij}$, for $i\neq j$, the two items above imply that every interval contributing to $K_2$ is contained in some interval contributing to $ K'_1$. Taking the union over all ordered pairs $(i,j)$ with $i\neq j$ yields $K_2\subset  K'_1$ implying $K_2\subset  K_1$.

\end{proof}

\section{Lower bounds and upper bounds}\label{sec:lbub}

The localization sets \(K_1\) and \(K_2\) constrain each complementarity eigenvalue, which is stronger than giving only global lower or upper bounds. For practical screening and benchmarking, it is often enough to control the extremes of the complementarity spectrum. We therefore extract computable bounds for the largest and the smallest complementarity eigenvalues, first in the one-row case and then in the two-row case, as stated in Corollaries~\ref{cor:one-row-bounds} and~\ref{cor:two-row-bounds}.

\begin{corollary}[One-row bounds]\label{cor:one-row-bounds}
Let $\A,\B\in\mathbb S_n$, and assume that $\B$ is positive definite and strictly diagonally dominant.  
Let $\lambda$ be a complementarity eigenvalue of $(\A,\B)$. Then,
\begin{equation*}
\min\!\left\{
 \min_{\,i\in[n]}\; \frac{\,a_i^{-}\,}{\,b_i^{-}\,},\; \min_{\,i\in[n]}\;
\frac{\,a_i^{-}\,}{\,b_i^{+}\,}
\right\} \leq \lambda \leq 
\max\!\left\{
\max_{\,i\in[n]}\;\frac{\,a_i^{+}\,}{\,b_i^{-}\,},\;\max_{\,i\in[n]}\;
\frac{\,a_i^{+}\,}{\,b_i^{+}\,}
\right\}.
\end{equation*}
\end{corollary}
\medskip

\begin{corollary}[Two-row bounds]\label{cor:two-row-bounds}
Let $\A,\B\in\mathbb S_n$. Assume that $\B$ is positive definite and strictly diagonally dominant, and $\A$ is copositive. Let $\lambda$ be a complementarity eigenvalue of $(\A,\B)$. Then,
\begin{equation*}
\min_{\substack{i,j\in[n]\\ i\ne j}}\;
\, 
\frac{ s_{ij}^{-}(\A,\B) - \sqrt{ s_{ij}^-(\A,\B)^2- \, 4\, b_{ij}^+\,\rowmm{a}{i}{j}}
}{ 2\, b_{ij}^{+} }
 \leq \lambda \le\;
\max_{\substack{i,j\in[n]\\ i\ne j}}\;
\, 
\frac{ s_{ij}^{+}(\A,\B) +
\sqrt{s_{ij}^+(\A,\B)^2- \, 4\, b_{ij}^-\,\rowpp{a}{i}{j}}
}{ 2\, b_{ij}^{-} }.
\end{equation*}
\end{corollary}

It is possible to interpret both bounds above as the \emph{convex hull} of $K_1$ and $K_2$. In other words, Corollaries~\ref{cor:one-row-bounds} and~\ref{cor:two-row-bounds} can be rewritten as $\lambda \in \operatorname{conv}(K_1)$ and $\lambda \in \operatorname{conv}(K_2)$, respectively.

\begin{example}\label{Example1}
    Consider the matrices in $\R^3$ given by
    \[
    \A =
    \begin{pmatrix}
    14 & 1  & 1 \\
     1 & 11 & -2\\
     1 & -2 & 13
    \end{pmatrix},
    \quad
    \B =
    \begin{pmatrix}
    6 & 0 & 0\\
    0 & 10 & 2\\
    0 & 2  & 10
    \end{pmatrix}.
    \]
\end{example}

Note that both matrices are SDDs with positive diagonals, so they are both PD. In particular, $\A$ is copositive. Then, the assumptions of Theorem~\ref{thm:one-row-theorem} and Theorem~\ref{thm:two-row-theorem} hold.

As described in the introduction, we can enumerate complementarity eigenpairs by supports $S\subset\{1,2,3\}$: for each nonempty $S$, candidates $\lambda$ are the generalized eigenvalues of $(\A_{SS},\B_{SS})$ with $\mathbf x_S\ge \mathbf 0$, then verified by the off-support inequalities $(\A_{iS}-\lambda \B_{iS})\x\ge \mathbf{0}$ when $i \notin S$. 
Feasible supports are $\{1\}$, $\{1,2\}$, $\{1,3\}$, $\{2,3\}$ and $\{1,2,3\}$, and the complementarity spectrum is
\[
\Pi \approx \left\{0.822,\;2.333,\;2.347,\;2.349,\;2.352\right\}.
\]

\begin{enumerate}[(i)]
%    \item \textbf{Two-row bounds.}
\item  The one-row enclosure from Corollary~\ref{cor:one-row-bounds} gives
\[
\Pi \subset \operatorname{conv}(K_1) = \left[\dfrac34,\dfrac83\right]
\approx  \left[0.750,\;2.667\right].
\]

\item The two-row enclosure from Corollary~\ref{cor:two-row-bounds} reduces to
\[
\Pi \subset \operatorname{conv}(K_2) = \left[\dfrac{31-\sqrt{127}}{24},\ \dfrac{109+\sqrt{1081}}{60}\right]\approx \left[0.822,\;2.365\right].
\]
As shown in Theorem~\ref{thm:2row-in-1row}, we have that $ \operatorname{conv}(K_2) \subset \operatorname{conv}(K_1)$.

%\item \textbf{One-row bounds.} 

\end{enumerate}

Let us compare this with the spectrum $[\mu_{\min}(\mathbf A,\mathbf B),\mu_{\max}(\mathbf A,\mathbf B)]$ of the classical generalized eigenvalues.
%illustrate the role of classical generalized eigenvalues on the simple $3\times3$ instance from Example~\ref{Example1}. 
The smallest and largest generalized eigenvalues of two symmetric matrices $(\A,\B)$ and with $\B$ being PD can be defined as the solution of minimization and maximization problems of the generalized Rayleigh quotient~\eqref{eq:Rayquo} for $\x\neq \mathbf 0$ (see Sect. A.5.3 in \cite{Boyd2004}), that is 
% \sd{imo, we could be more explicit about the link between Rayleigh quotient and GeiP, and also about this direct computation ?}
% \as{A direct computation of $\mu_{\min}$ and $\mu_{\max}$ from the Rayleigh quotient~\cite{HornJohnson2013} %of the generalized spectrum of $(\mathbf A,\mathbf B)$ 
% gives}
\begin{align}\label{eq:minmaxmu}
\mu_{\min}(\mathbf A,\mathbf B):=\min_{\mathbf x\neq \mathbf 0}\frac{\mathbf x^\top \mathbf A\,\mathbf x}{\mathbf x^\top \mathbf B\,\mathbf x},\quad
\mu_{\max}(\mathbf A,\mathbf B):=\max_{\mathbf x\neq \mathbf 0}\frac{\mathbf x^\top \mathbf A\,\mathbf x}{\mathbf x^\top \mathbf B\,\mathbf x}.
\end{align}
Then, we have $\Pi\subset[\mu_{\min}(\mathbf A,\mathbf B),\mu_{\max}(\mathbf A,\mathbf B)]\approx[0.804,\,2.352]$. Hence, 
in this instance, the generalized spectrum is also a localization set for the complementarity eigenvalues, which
is contained in $ \operatorname{conv}(K_1)$ but only overlaps $ \operatorname{conv}(K_2)$.
%interval between the extreme generalized eigenvalues is a localization set on this instance. In particular, 
%the length of $K_2$ is slightly tighter than $[\mu_{\min},\mu_{\max}]$ (on the other hand, the length of $K_1$ is looser).

\section{Comparison with the generalized spectrum}\label{section5}

In this section, we show that this last remark %phenomenon 
is not incidental: the smallest and largest generalized eigenvalues always provide lower and upper bounds for the complementary eigenvalues of~\eqref{eq:EiCP}, and they are not comparable with the bounds of Section~\ref{sec:lbub}.

\begin{proposition}[Generalized spectral localization set]
Let $\mathbf A,\mathbf B\in \mathbb{S}_n$ and assume that $\mathbf B$ is positive definite. If $\lambda$ is a complementarity eigenvalue of $(\A,\B)$ and $\mu_{\min}(\mathbf A,\mathbf B)$ and $\mu_{\max}(\mathbf A,\mathbf B)$
are the minimum and maximum generalized eigenvalues of $(\A,\B)$, then
$$\lambda\in [\mu_{\min}(\mathbf A,\mathbf B),\mu_{\max}(\mathbf A,\mathbf B)]:=\Gamma$$
%then $\mu_{\min}(\mathbf A,\mathbf B)\ \le\ \lambda\ \le\ \mu_{\max}(\mathbf A,\mathbf B)$. Therefore $\Gamma := [\mu_{\min}(\mathbf A,\mathbf B),\mu_{\max}(\mathbf A,\mathbf B)]$ localizes every complementarity eigenvalue of $(\mathbf A,\mathbf B)$.
\end{proposition}

\begin{proof}
Let $(\mathbf x,\lambda)$ be a solution of EiCP~\eqref{eq:EiCP}. The orthogonality constraint~
\eqref{eq:EiCP_c} %${\x^\top}{(\A-\lambda \B)\x}=0$ 
yields $\lambda~=~\frac{\mathbf x^\top \mathbf A\,\mathbf x}{\mathbf x^\top \mathbf B\,\mathbf x}$, which is well defined as $\mathbf B$ is $\operatorname{PD}$ and $\x\neq \mathbf 0$. By using the definitions in~\eqref{eq:minmaxmu}, we get the result.

\end{proof}

%There is no universal inclusion relation between $\Gamma$ and our sets $K_1$ and $K_2$. 
To show that dominance does not hold between $\Gamma$ and the sets $K_1$ and $K_2$, we rely on the following lemma, which characterizes the generalized eigenvalues when $\A$ and $\B$ commute. We then derive two families of problems where the generalized spectrum is either strictly looser (Proposition~\ref{case_loc}) or strictly tighter (Proposition~\ref{case_eig}) than our localization sets.
%Proposition~\ref{case_loc} presents a structured family in which the generalized spectrum is strictly looser than our localization sets. By contrast, Proposition~\ref{case_eig} exhibits another family where the generalized spectrum is strictly tighter.
 
\begin{lemma}\label{lemma_commutation}
Let $\mathbf{A},\mathbf{B}\in \mathbb{S}_n$, assume that $\mathbf{B}$ is positive definite and $\mathbf{A}\mathbf{B}=\mathbf{B}\mathbf{A}$. 
Let $\mu_1^{\mathbf{A}},\dots,\mu_n^{\mathbf{A}}$ be the standard eigenvalues of $\mathbf{A}$ and let $\mu_1^{\mathbf{B}},\dots,\mu_n^{\mathbf{B}}>0$ be those of $\mathbf{B}$, all counted with multiplicity. 
Then, the generalized eigenvalues $\mu_1, \dots, \mu_n$ of the pair $(\mathbf{A},\mathbf{B})$ are exactly the $n$ ratios
\[
\mu_1={\mu^{\mathbf{A}}_{\sigma(1)}}/{\mu^{\mathbf{B}}_{\tau(1)}},\ \ldots,\ 
\mu_n={\mu^{\mathbf{A}}_{\sigma(n)}}/{\mu^{\mathbf{B}}_{\tau(n)}}
\]
for some permutations $\sigma,\tau$ of $[n]$. 
%Equivalently, after 
By indexing a common orthonormal eigenbasis, one can take $\sigma$ and $\tau$ as the identity, to obtain $\mu_i=\mu^{\mathbf{A}}_i/\mu^{\mathbf{B}}_i$ for $i\in[n]$.
\end{lemma}

\begin{proof}
By the fact that commuting real symmetric matrices are simultaneously diagonalizable by an orthogonal matrix (see Thm. 2.5.5 in \cite{HornJohnson2013}), there exists an orthogonal $\mathbf{V}\in\mathbb{R}^{n\times n}$ such that
\[
\mathbf{V}^\top \mathbf{A}\,\mathbf{V}=\mathrm{diag}(\mu^{\mathbf{A}}_1,\ldots,\mu^{\mathbf{A}}_n), 
\quad
\mathbf{V}^\top \mathbf{B}\,\mathbf{V}= \mathrm{diag}(\mu^{\mathbf{B}}_1,\ldots,\mu^{\mathbf{B}}_n).
\]
Consider the equation $\mathbf{A}\mathbf{x}=\mu\,\mathbf{B}\mathbf{x}$ for any generalized eigenvalue $\mu$. 
Setting $\mathbf{y}=\mathbf{V}^\top\mathbf{x}$ yields
\[
(\mathbf{V}^\top \mathbf{A}\,\mathbf{V})\mathbf{y}=\mu\,(\mathbf{V}^\top \mathbf{B}\,\mathbf{V})\mathbf{y}
\ \Longleftrightarrow\
(\mu^{\mathbf{A}}_i-\mu\,\mu^{\mathbf{B}}_i)\,y_i=0\quad\text{for all }i\in[n].
\]
If $\mathbf{y}\neq\mathbf{0}$, then some $y_i\neq 0$ enforces $\mu=\mu^{\mathbf{A}}_i/\mu^{\mathbf{B}}_i$. 
Conversely, for each $i\in[n]$, taking $\mathbf{y}=\mathbf{e}_i$ (hence, $\mathbf{x}=\mathbf{V}\mathbf{e}_i$) produces a nonzero solution $\mu_i=\mu^{\mathbf{A}}_i/\mu^{\mathbf{B}}_i$. 
Thus the generalized eigenvalues are precisely the $n$ ratios $\mu^{\mathbf{A}}_i/\mu^{\mathbf{B}}_i$ (counted with multiplicity).

\end{proof}

\begin{proposition}[A family where $K_1 = K_2 \subsetneq \Gamma$]\label{case_loc}
Fix $n\ge 2$ and $\varepsilon>1$. Let $\mathbf E=\mathbf e\,\mathbf e^\top$, and define
\[
\mathbf A:=\mathbf E+\varepsilon\,\mathbf I,\quad
\mathbf B:=((n-1)+\varepsilon)\,\mathbf I-\mathbf E.
\]
Then, $\mathbf B$ is $\operatorname{SDD}$ and $\operatorname{PD}$, and $\mathbf A$ is copositive (so Theorems~\ref{thm:one-row-theorem} and~\ref{thm:two-row-theorem} apply). Moreover,
\[
\operatorname{conv}(K_1) = \operatorname{conv}(K_2) =\left[ \frac{1+\varepsilon}{\,n-2+\varepsilon\,}\ ,\ \frac{n+\varepsilon}{\,\varepsilon-1\,}\right],
\quad
[\mu_{\min}(\mathbf A,\mathbf B),\mu_{\max}(\mathbf A,\mathbf B)]
=\left[ \frac{\varepsilon}{\,n-1+\varepsilon\,}\ ,\ \frac{n+\varepsilon}{\,\varepsilon-1\,}\right].
\]
Hence, $K_1 =K_2 \subsetneq [\mu_{\min}(\mathbf A,\mathbf B),\mu_{\max}(\mathbf A,\mathbf B)]$.
\end{proposition}

\begin{proof}
$\mathbf B$ is $\operatorname{SDD}$ because $b_{ii}=n-2+\varepsilon> \sum_{j\ne i}|b_{ij}|=n-1$ for $\varepsilon>1$, and $\B$ is symmetric with positive diagonal elements, thus $\mathbf B$ is $\operatorname{PD}$. Moreover, $\A$ is nonnegative,
%is $\operatorname{PD}$ \sd{[how ?]}, 
thus it is copositive. The matrices $\mathbf A$ and $\mathbf B$ commute (as $\A\B=\B\A=\delta\I-\mathbf{E}$ for some constant $\delta$) and from Lemma \ref{lemma_commutation}
we can compute the values of $\mu_{\min}$ and $\mu_{\max}$ by calculating the largest and smallest eigenvalues of $\A$ and $\B$. 
To show that, we first recall the spectral properties of $\mathbf{E}$.  
Since $\mathbf{E}=\mathbf{e}\mathbf{e}^{\mathsf T}$, it follows that $\mathbf{E}\mathbf{e}=n\mathbf{e}$ and $\mathbf{E}\mathbf{v}=0$ for all $\mathbf{v}$ such that $\mathbf{v}^{\mathsf T}\mathbf{e}=0$. Hence, $\mathbf{E}$ has eigenvalue $n$ associated with the eigenvector $\mathbf{e}$, and eigenvalue $0$ with multiplicity $n-1$ associated with any vector orthogonal to $\mathbf{e}$. Using this decomposition of $\mathbf{E}$, we can determine the eigenvalues of $\mathbf{A}$ and $\mathbf{B}$ below in (i) and (ii).

\begin{enumerate}[(i)]
    \item For $\mathbf{A}=\mathbf{E}+\varepsilon\mathbf{I}$, we have
\[
\mathbf{A}\mathbf{e}
=(\mathbf{E}+\varepsilon\mathbf{I})\mathbf{e}
=(n+\varepsilon)\mathbf{e},
\quad
\mathbf{A}\mathbf{v}
=(\mathbf{E}+\varepsilon\mathbf{I})\mathbf{v}
=\varepsilon\mathbf{v}
\quad \text{for every $\mathbf{v}$ such that } \mathbf{v}^{\mathsf T}\mathbf{e}=0.
\]
Therefore, from the eigenvalue equations above, the classical spectrum of $\A$ is equal to $\{n+\varepsilon,\;\underbrace{\varepsilon, \dots, \varepsilon}_{n-1 \text{ times}}\}$ and then $\mu_{\min}^\mathbf{A} = \varepsilon$, and $\mu_{\max} ^\mathbf{A} = n+\varepsilon$.

\item For $\mathbf{B}=\big((n-1)+\varepsilon\big)\mathbf{I}-\mathbf{E}$,
\[
\mathbf{B}\mathbf{e}
=\big((n-1)+\varepsilon\big)\mathbf{e}-\mathbf{E}\mathbf{e}
=\big((n-1)+\varepsilon-n\big)\mathbf{e}
=(\varepsilon-1)\mathbf{e},
\]
and for any $\mathbf{v}$ such that $\mathbf{v}^{\mathsf T}\mathbf{e}=0$,
\[
\mathbf{B}\mathbf{v}
=\big((n-1)+\varepsilon\big)\mathbf{v}-\mathbf{E}\mathbf{v}
=\big((n-1)+\varepsilon\big)\mathbf{v}.
\]
Similarly, the spectrum of $\B$ is given by $\{\varepsilon-1,\;\underbrace{n-1+\varepsilon, \dots, \; n-1+\varepsilon}_{{n-1 \text{ times}}}\}$ and thus we have $\mu_{\min} ^\mathbf{B} = \varepsilon-1$ and $\mu_{\max} ^\mathbf{B} = n-1+\varepsilon$.
%by Lemma \ref{lemma_commutation}.
\end{enumerate}
Hence, $\mu_{\min} (\mathbf A,\mathbf B)=\frac{\varepsilon}{n-1+\varepsilon}$ and $\mu_{\max} (\mathbf A,\mathbf B)= \frac{n+\varepsilon}{\varepsilon-1}$.

For $K_2$, substitute $r_i^-(\mathbf A)=0$, $r_i^+(\mathbf A)=n-1$, $r_i^-(\mathbf B)=n-1$, $r_i^+(\mathbf B)=0$ into polynomial expressions of Lemma~\ref{lem:pol}.
%Theorem~\ref{thm:two-row-theorem}. 
This yields
\[
P^{\, \operatorname{low}}_{ij}(y)=\bigl((n-2+\varepsilon)y-(1+\varepsilon)\bigr)^2,
\quad
P^{\, \operatorname{up}}_{ij}(y)=\bigl((n-2+\varepsilon)y-(1+\varepsilon)\bigr)^2-(n-1)^2(1+y)^2,
\]
so the intervals in $K_2$ are all equal to $\left[ \frac{1+\varepsilon}{\,n-2+\varepsilon\,}\ ,\ \frac{n+\varepsilon}{\,\varepsilon-1\,}\right]$ independently of $(i,j)$. 
From Corollary~\ref{cor:onecop}, note that all intervals in $K_1$ also coincide with this. 
%Finally,
The upper bound coincides with $\mu_{\max}$, but the lower bound strictly improves upon $\mu_{\min}$, as:
\[
\frac{1+\varepsilon}{n-2+\varepsilon}-\frac{\varepsilon}{n-1+\varepsilon}
=\frac{(n-1)+2\varepsilon}{(n-2+\varepsilon)(n-1+\varepsilon)}>0,
\]
%which shows the strict improvement at the lower endpoint. 
%Moreover, the one-row endpoints from Theorem~\ref{thm:one-row-theorem} coincide with those of $K_2$ for this family, hence $K_1=K_2$.
\end{proof}

\begin{proposition}[A family where $\Gamma \subsetneq K_1 = K_2$]\label{case_eig}
Fix $n\ge 2$ and parameters $\beta>R>0$ and $c>0$. Let $\mathbf E=\mathbf e\,\mathbf e^\top$, set $\rho:=R/(n-1)$, and define
\[
\mathbf B:=\beta\,\mathbf I+\rho\,(\mathbf E-\mathbf I),\quad
\mathbf A:=c\,\mathbf B.
\]
Then $\mathbf B$ is $\operatorname{SDD}$ and $\operatorname{PD}$, and $\mathbf A$ is copositive (so Theorems~\ref{thm:one-row-theorem} and~\ref{thm:two-row-theorem} apply). Moreover
\[
\operatorname{conv}(K_1) = \operatorname{conv}(K_2) = \Bigl[\frac{c\,\beta}{\beta+R}\,,\ \frac{c\,(\beta+R)}{\beta}\,\Bigr],
\quad
[\mu_{\min}(\mathbf A,\mathbf B),\mu_{\max}(\mathbf A,\mathbf B)]
=\{c\}.
\]
Hence, $[\mu_{\min}(\mathbf A,\mathbf B),\mu_{\max}(\mathbf A,\mathbf B)] \subsetneq K_1=K_2$.
\end{proposition}
%\sd{[Finally it is not that bad to exhibit a very specific example for this unfavorable case (it would have been more arguable for the favorable case)]}

\begin{proof}
Each row of $\mathbf B$ has diagonal entry $\beta>0$ and the sum of the off-diagonal elements is $(n-1)\rho$, which is $R$ by definition. Also, since $\beta>R$, then $\mathbf B$ is $\operatorname{SDD}$ and then $\operatorname{PD}$. The matrix $\mathbf A$ is copositive as $c>0$ and $\B$ is PD. Therefore Theorems~\ref{thm:one-row-theorem} and~\ref{thm:two-row-theorem} are valid.
For the generalized spectrum, $\mathbf B^{-1}\mathbf A=c\,\mathbf I$, hence %$\A=c\B$ and $\B$ commute, hence,
%from Lemma~\ref{lemma_commutation},
$\mu_{\min}(\mathbf A,\mathbf B)=\mu_{\max}(\mathbf A,\mathbf B)=c$, and therefore the interval $\Gamma$ is the point $\{c\}$.

To compute $K_2$, note the row-sum quantities are constant across $i\in [n]$:
\[
a_{ii}=c\beta,\quad r_i^+(\mathbf A)=cR,\quad r_i^-(\mathbf A)=0,\quad
b_{ii}=\beta,\quad r_i^+(\mathbf B)=R,\quad r_i^-(\mathbf B)=0.
\]
Substituting in the polynomials of Lemma~\ref{lem:pol}
%Substituting in the two-row polynomials of %Theorem~\ref{thm:two-row-theorem} 
yields, for any $i\neq j$,
\[
P^{\, \operatorname{low}}_{ij}(y)
%=(\beta^2-R^2)y^2-2c\beta^2 y+c^2\beta^2
=(y\beta -c\beta )^2-(yR)^2
\mbox{ and }
%=\bigl((\beta+R)y-c\beta\bigr)\bigl((\beta-R)y-c\beta\bigr),
%\]
%\[
P^{\, \operatorname{up}}_{ij}(y)
%=\beta^2 y^2-2c\beta^2 y+c^2(\beta^2-R^2)
=(y\beta -c\beta )^2-(cR)^2.
%=\bigl(\beta y-c(\beta+R)\bigr)\bigl(\beta y-c(\beta-R)\bigr).
\]
Hence,
\[
\min\{y:\ P^{\, \operatorname{low}}_{ij}(y)=0\}=\frac{c\,\beta}{\beta+R},
\quad
\max\{y:\ P^{\, \operatorname{up}}_{ij}(y)=0\}=\frac{c\,(\beta+R)}{\beta},
\]
independently of $(i,j)$, so $K_2=[\,c\,\beta/(\beta+R),\ c\,(\beta+R)/\beta\,]$. Using the formula in Corollary~\ref{cor:onecop} gives exactly the same endpoints, hence $K_1=K_2$.

Finally, since $\beta/(\beta+R)<1<(\beta+R)/\beta$ for $\beta>R>0$, we have $\frac{c\,\beta}{\beta+R}\;<\ c\;<\ \frac{c\,(\beta+R)}{\beta}$, so $\{c\}\subsetneq K_1=K_2$ strictly.

\end{proof}

\section{Discussion}\label{secd}

We can unify the formula of $K_1$ in Corollary~\ref{cor:onecop} and of $K_2$ in Lemma~\ref{lem:pol}, under their common assumptions ($\A$ copositive and $\B$ SDD and PD) as follows. For $m=1$ and $m=2$,
%\begin{remark}[Unified statement]
%Let $\mathbf{A},\mathbf{B}\in\mathbb{S}_n$. Assume that $\B$ is positive definite and strictly diagonally dominant, and $\A$ is copositive.  Let $\lambda$ be a complementarity eigenvalue of $(\mathbf{A},\mathbf{B})$. 
\begin{align}
K_m = &\bigcup_{\substack{S\subset[n]\\
|S|=m}}\,[\,\max\{0, \, \min\{y\in\mathbb{R}\mid P^{\, \mathrm{low}}_{S}(y)=0\}\},\,\max\{y\in\mathbb{R}\mid P^{\, \mathrm{up}}_{S}(y)=0\}\,]\\
\text{given} \quad &P^{\mathrm{up}}_{S}(y)
:=\prod_{i\in S}\bigl(b_{ii}y-a_{ii}\bigr)\;-\;\prod_{i\in S}\bigl(r_i^{+}(\mathbf{A})+y\,r_i^{-}(\mathbf{B})\bigr), \label{eq:PupS}\\
&P^{\mathrm{low}}_{S}(y)
:=\prod_{i\in S}\bigl(a_{ii}-b_{ii}y\bigr)\;-\;\prod_{i\in S}\bigl(r_i^{-}(\mathbf{A})+y\,r_i^{+}(\mathbf{B})\bigr) \quad \forall y \in \mathbb{R}.\label{eq:PlowS}
\end{align}
%Then, we have $\ \lambda\ \in\ \bigcup_{\substack{\emptyset\neq S\subset[n]\\
%|S|\le 2}}\,[\,\min\{y\in\mathbb{R}\mid P^{\, \mathrm{low}}_{S}(y)=0\},\,\max\{y\in\mathbb{R}\mid P^{\, \mathrm{up}}_{S}(y)=0\}\,]\;=\;K_{\lvert S\rvert}.$

%In particular, $|S|=1$ reproduces the one-row set $K_1$, and $|S|=2$ reproduces the two-row set $K_2$ (the polynomials obtained by \eqref{eq:PupS} and \eqref{eq:PlowS} coincides with those of Theorem \ref{thm:two-row-theorem}). 
%\end{remark}

In these two cases, $m$ denotes the number of rows considered to derive each interval. One may ask if this formula extends to $m\ge 3$. Unfortunately, Example~\ref{Example1} provides a counter-example for $m=3$.
%On the other hand, the direct extension to $|S|=3$ fails. Consider the same instances from Example~\ref{Example1}.
Indeed, defining \eqref{eq:PupS} and \eqref{eq:PlowS} for the triple $S=\{1,2,3\}$ gives
\[
P^{\mathrm{up}}_{S}(y)=(6y-14)(10y-11)(10y-13)-2,\quad
P^{\mathrm{low}}_{S}(y)=(14-6y)(11-10y)(13-10y).
\]
The smallest real root of $P^{\mathrm{low}}_{S}$ is $1.1$, yet the complementarity eigenvalue 
$\lambda\approx0.822\in\Pi$ lies to its left. 
The largest real root of $P^{\mathrm{up}}_{S}$ is $\approx2.336$, yet 
$\lambda \approx2.347\in\Pi$ lies to its right (indeed $P^{\mathrm{up}}_{S}(2.347)\approx 8.5>0$). 
The question of finding a hierarchy of localization sets based on the size of the considered submatrices remains thus open.
%Therefore, the rule “take the smallest/largest roots of the productor polynomials as lower/upper bounds’’ does \emph{not} hold for $|S|\ge 3$.

\section{Conclusions} \label{secc}

We analyzed the symmetric EiCP$(\A,\B)$ with $\B$ being positive definite and strictly diagonally dominant. Extending the He-Liu-Shen enclosures to a more general positive definite matrix $\B$, we present Gershgorin-type localization sets $K_1$ and $K_2$, the latter improving upon the former, assuming $\A$ copositive. The bounds on these sets provide inexpensive approximations of the smallest and largest complementarity eigenvalues, as do the smallest and largest generalized eigenvalues, and no dominance exists between these bounds.
%Finally, every complementarity eigenvalue lies in $[\mu_{\min}(\A,\B),\mu_{\max}(\A,\B)]$, but neither this interval nor our enclosures generally contains the other.

Our results also suggest several directions for future research.
In an algorithmic use, the localization sets or their bounds provide fast certificates for candidates of complementarity eigenpairs. The discussion opens the question of generalizing the approach to multi-row coupling. Finally, this approach could be adapted to EiCP variants, such as the tensor eigenvalue complementarity problem.
%\begin{enumerate}
%\item \emph{Richer row couplings:} investigate multirow variants that preserve enclosure validity (counterexamples show that naive $\lvert S\rvert\ge 3$ products do not).
%\item  \emph{Other eigenproblems:} adapt the row-sum approach to EiCP variants, for example to the tensor eigenvalue complementarity problem \cite{FNZh}.
%\item  \emph{Algorithmic use:} integrate $K_1$ and $K_2$ in solvers as cheap certificate for candidates to be complementarity eigenpairs.
%\end{enumerate}
%\printbibliography
\vskip 6mm
\noindent{\bf Acknowledgements} 

\noindent The authors thank Jorma K. Merikoski, Pentti Haukkanen, Timo Tossavainen, and Mika Mattila for the useful conversations, in particular regarding Lemma~\ref{lemma_commutation}.

\bibliographystyle{plain} 
\bibliography{EiCP.bib} 

\end{document}